\documentclass[12pt]{article}
\usepackage[margin=1.9cm]{geometry}
\usepackage{float}

\usepackage[pagebackref, hidelinks]{hyperref}
\hypersetup{
colorlinks=true, 
citecolor = blue, 
linkcolor = blue, 
urlcolor = blue, 
anchorcolor = blue
}

\usepackage{amsthm}
\usepackage{amssymb}
\usepackage{tikz}

\usepackage{mathtools}

\usepackage{enumitem}
\usepackage{scalefnt}

\usepackage{paratex/commands}

\title{The complexity of pinning simple multiloops}

\author{Eric Seo, Christopher-Lloyd Simon, Ben Stucky}
\date{\today}

\begin{document}

\maketitle

\begin{abstract}

A multiloop with $s\in \N$ strands is a generic immersion \(\gamma\colon \sqcup_1^s \mathbb{S}^1 \looparrowright \Sigma\) of the union of $s$ circles into a surface $\Sigma$, considered up to homeomorphisms. 
A pinning set of $\gamma$ is a set of points \(P\subset \Sigma\setminus \operatorname{im}(\gamma)\), such that in the punctured surface \(\Sigma \setminus P\), the immersion \(\gamma\) has the minimal number of double points in its homotopy class.
Its pinning number \(\varpi(\gamma)\) is the minimum cardinal of its pinning sets.

In any fixed orientable surface $\Sigma$, the pinning problem which given a multiloop $\gamma$ and $k\in \N$ decides whether $\varpi(\gamma)\le k$ has been show to be \textsf{NP}-complete, even in restrictions to loops (with $s=1$ strand).
In this work we study the complexity of the pinning problem in restriction to multiloops whose strands are \emph{simple} (embedded circles).

We show that in any fixed oriented surface $\Sigma$, the problem is in \textsf{P} when $s\leq 3$ and \textsf{NP}-complete when $s\geq 20$, and present some follow-up questions and conjectures.
\end{abstract}

\begin{figure}[H]
\centering
\vspace{-2.9cm}
\rotatebox{-45}{\scalebox{0.5}{\input{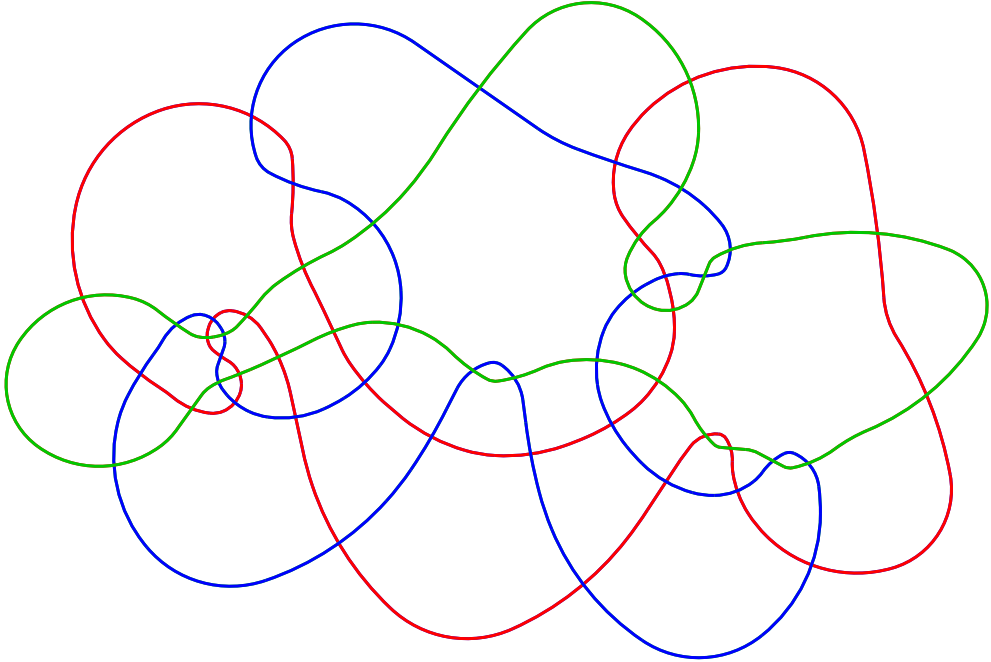}}}
\vspace{-2.9cm}
\caption{A planar simple multiloop with $3$ strands, drawn with a tool created by Ryan Pham and the second and third author using a circle-packing algorithm from PADS \cite{collins_stephenson_2003,eppsteinPADS,Pham_Computing_mobidisc_formulae_2026}.}
\label{fig:planar_multiloop}
\end{figure}

\newpage

\renewcommand{\contentsname}{Plan of the paper}
\setcounter{tocdepth}{2}
\tableofcontents

\section{Introduction}

\subsection{Key definitions and previous work}

\begin{definition}[multiloop, simple multiloop, pinning set, pinning number]

A \emph{multiloop} with $s\in \N$ \emph{strands} in a smooth surface $\Sigma$ is a generic immersion \(\gamma\colon \sqcup_1^s \mathbb{S}^1 \looparrowright \Sigma\) of the union of $s$ circles into $\Sigma$, considered up to diffeomorphisms of the source and the target.

In particular, it has a finite number of multiple points and those are all transverse double-points: we denote that \emph{number of double-points} by $\#\gamma$. A multiloop is called \emph{simple} when each of its strands is embedded in $\Sigma$.

The connected components of $\Sigma\setminus\operatorname{im}(\gamma)$ will be called \emph{regions} of $\gamma$ and we denote by $\mathcal{R}(\gamma)$ the set of regions.
For a subset of regions $P\subset \mathcal{R}(\gamma)$, we define the \emph{self-intersection number} $\si_P(\gamma)$ of $\gamma$ in $\Sigma\setminus P$ as the minimal number of double-points of multiloops in $\Sigma\setminus P$ homotopic to $\gamma$.

We say that $P$ is a \emph{pinning set} for $\gamma$ when $\#\gamma = \si_P(\gamma)$, namely when $\gamma$ is \emph{taut} in $\Sigma \setminus P$. 
The collection of pinning sets of \(\gamma\) forms the \emph{pinning ideal} \(\mathcal{PI}(\gamma)\), that is a poset under inclusion which is absorbing under union.
Among the \emph{minimal} pinning sets, some are \emph{optimal} in the sense that they minimize the cardinal function: this minimum cardinal is called the \emph{pinning number} \(\varpi(\gamma)\). 
\end{definition}

\begin{remark}[no orientations]
    By definition, a multiloop has non-oriented strands and its ambient surface is not endowed with a preferred orientation (even when it is orientable), since the diffeomorphisms may reverse any of these orientations.
\end{remark}

\begin{example}[empirical structure of pinning ideals]
    A multiloop may have minimal pinning sets of different cardinalities as demonstrated in Figure \ref{fig:simple_pinningideal}. The second and third author used the results of \cite{simon_stucky_2025pinningidealmultiloop} and some functionalities from \texttt{plantri} \cite{plantri} and \texttt{SnapPy} \cite{SnapPy} to enumerate the smallest multiloops in $\mathbb{S}^2$ and compute and plot their pinning ideals. In particular we did this for \href{https://christopherlloyd.github.io/LooPindex/simple.html}{the smallest simple multiloops}. The full results are available in the \href{https://christopherlloyd.github.io/LooPindex/}{LooPindex online catalog} \cite{Simon-Stucky_LooPindex_2024}. 
\end{example}

\begin{figure}[H]
    \centering
    \scalebox{0.5}{\input{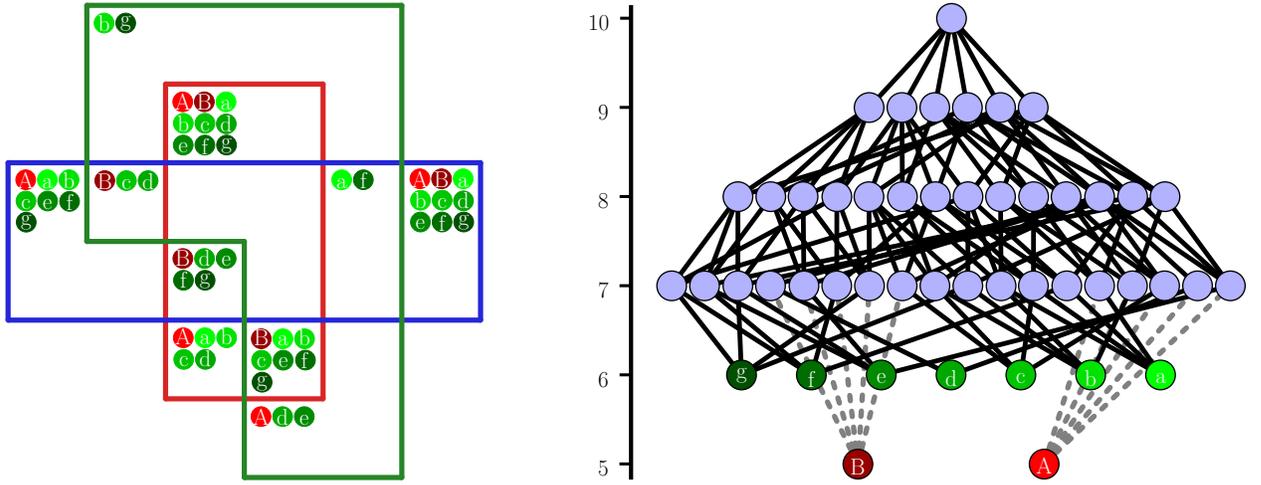}}
    \caption{Left: The simple multiloop \href{https://christopherlloyd.github.io/LooPindex/multiloops/12\%5E3_56.html}{$12^3_{56}$} has number pinning number $5$, two optimal pinning sets (denoted by capital letters), and seven other minimal pinning sets of size $6$ (denoted by lowercase letters). Right: The part of its pinning ideal generated by unions of minimal pinning sets.}
    \label{fig:simple_pinningideal}
\end{figure}

\begin{definition}[pinning problem]
    Fix a surface $\Sigma$; $s\in \N$.
    The \emph{pinning problem} $\mathsf{Pin}(\Sigma, s)$ has

    \begin{itemize}[align=left, noitemsep]
    \item[Instance:] A combinatorial encoding of a multiloop $\gamma\colon \sqcup_{i=1}^s \Sph^1\looparrowright \Sigma$ (see below) and $k\in \Z$.
    \item[Question:] Is  $\varpi(\gamma)\le k$?
\end{itemize}
\end{definition}

\begin{remark}[filling multiloops]
A multiloop $\gamma$ is \emph{filling} when all of its regions $\mathcal{R}(\gamma)$ have genus zero (namely are homeomorphic to disks with a certain number of punctures). It is enough to restrict to multiloops that are filling when studying the complexity of $\mathsf{Pin}(\Sigma, s)$ because any region with positive genus may be replaced by a disk with a puncture.
\end{remark}

\begin{definition}[combinatorial encodings of multiloops]
Our complexity results depend on the existence of combinatorial encodings of multiloops in surfaces. 
There are several ways to do this which lead to equivalent notions of complexity.
For filling multiloops, one may do this using maps and their encodings as permutations as in \cite{simon_stucky_2025pinningidealmultiloop}. 
The size of a $\mathsf{Pin}(\Sigma, s)$ problem instance $(\gamma,k)$ in memory (that is, the size of the input) is then proportional to $\#\gamma+\abs{k}$.

\end{definition}

The second and third author show in \cite{simon_stucky_2025pinningidealmultiloop} that for every fixed orientable $\Sigma$ and $s\geq 1$, the pinning problem $\mathsf{Pin}(\Sigma, s)$ is \textsf{NP}-complete, by a reduction from the graph-theoretic decision problem \textsf{Planar Vertex Cover} (see Definition \ref{def:vc-prob}).

\subsection{Summary of results}

In this work we consider the restriction of the pinning problem to the class of simple multiloops, and will show that its complexity exhibits a ``phase transition'' as the number of strands increases.

\begin{definition}[$\mathsf{SimplePin}$]
For a surface $\Sigma$ and $s\in \N$, the decision problem $\mathsf{SimplePin}(\Sigma,s)$ has 
\begin{itemize}[align=left, noitemsep]
    \item[Instance:] A simple multiloop $\gamma\colon \sqcup_{i=1}^s \Sph^1\looparrowright \Sigma$ and $k\in \Z$.
    \item[Question:] Is  $\varpi(\gamma)\le k$?
\end{itemize}
Recalling the formal language definition, we also let $\mathsf{SimplePin}(\Sigma)=\cup_{s=1}^\infty\mathsf{SimplePin}(\Sigma,s)$.
\end{definition}

A simple loop is trivially pinned by the empty set.
For simple multiloops with $2$ strands, it is an exercise (as we will see) to show that there is a unique minimal pinning set (hence a unique optimal pinning set) and it is computable in polynomial time, so the problem is in \textsf{P}.
Our first main result is that the pinning problem remains easy for simple multiloops with $s\le 3$ strands.

\begin{theorem}[in \textsf{P} for $\le 3$ strands]
\label{thm:3-easy}
For a fixed orientable surface $\Sigma$ and $s\leq 3$, the problem $\mathsf{SimplePin}(\Sigma,s)$ is in \textsf{P}.
\end{theorem}

\begin{proof}[Proof outline]
This is proved in Section \ref{subsec:3easy} by reducing the problem of bounding $\varpi(\gamma)$ for $\gamma\colon \sqcup_{i=1}^3 \Sph^1\looparrowright \Sigma$ to the \textsf{Vertex Cover} problem for bipartite graphs, which is in \textsf{P} because of its relation to matchings (via K\fH{o}nig's theorem and Dinitz' algorithm \cite{dinitz_1970,edmonds_karp_72,bondy1976graph}).
\end{proof}

In contrast our next main theorem implies, under the assumption \textsf{P} $\neq$ \textsf{NP}, that for any fixed orientable surface $\Sigma$ and integer $s\geq 20$, there is no polynomial algorithm which takes as input a simple multiloop $\gamma\colon \sqcup_{i=1}^s \Sph^1\looparrowright \Sigma$ and outputs $\varpi(\gamma)$.

\begin{theorem}[\textsf{NP}-complete for many strands]
\label{thm:20-hard}
For a fixed orientable surface $\Sigma$ and $s\geq 20$, the problem $\mathsf{SimplePin}(\Sigma,s)$ is \textsf{NP}-complete.
\end{theorem}

\begin{proof}[Proof outline]
Following the approach of \cite{simon_stucky_2025pinningidealmultiloop} for loops, this is proved in Section \ref{subsec:20hard} by reducing a restricted but still \textsf{NP}-hard version of \textsf{Vertex Cover} (\textsf{3-Connected Cubic Planar Vertex Cover} or \textsf{3C3PVC}, see Definition \ref{def:vc-prob}) to the problem of bounding $\varpi(\gamma)$ (for $\gamma\colon \sqcup_{i=1}^s \Sph^1\looparrowright \Sigma$ with $s\geq 20$). However, the construction is more delicate than \cite{simon_stucky_2025pinningidealmultiloop} as it requires connecting local strands together to create only $20$ strands globally in such a way that each strand is simple. We will make crucial use of a result of Dujmov\'{i}c, Eppstein, Suderman, and Wood that every cubic $3$-connected plane graph has a plane drawing (computable in polynomial time) in which every edge
has slope in $\{\frac{\pi}{4}, \frac{\pi}{2}, \frac{3\pi}{4}\}$, except for three edges on the outer face \cite{dumovic_et_al_few_slopes_07}. The number $20$ arises because we will need $4$ strands for each slope in $\{\frac{\pi}{4}, \frac{\pi}{2}, \frac{3\pi}{4}\}$, $2$ strands for each of the $3$ additional edges on the outer face, and $2$ additional strands to serve as ``anchors at infinity''. We will say a bit more about the possibility of using fewer strands in Section \ref{sec:conclusion}.
\end{proof}

\begin{remark}[restricting to $\R^2$ to establish hardness]
\label{rmk:plane_enough}
If $\Sigma_1$ is homotopic to a subsurface of an orientable surface $\Sigma_2$, then a hardness result for a pinning problem in $\Sigma_1$ implies the same hardness result for that pinning problem in $\Sigma_2$. So to prove that a pinning problem is \textsf{NP}-hard, it is enough to prove it for $\Sigma=\R^2$.
However, for easyness results (such as proving that a pinning problem is in  \textsf{P} or \textsf{NP}), we must deal with all surfaces.
\end{remark}

\subsection{Further directions of research}

\subsubsection{Hardness threshold}

Theorems \ref{thm:3-easy} and \ref{thm:20-hard} raise the question: For a fixed surface $\Sigma$, what is the largest $s\in \N$ such that $\mathsf{SimplePin}(\Sigma,s)$ is \textsf{P}?
We suspect that the answer to this question is $3$, and provide some rationale for the following conjecture in Section \ref{sec:4-strands}.

\begin{conjecture}[hardness starts at $s=4$]
\label{conj:4-hard}
For a fixed orientable surface $\Sigma$, the problem $\mathsf{SimplePin}(\Sigma,4)$ is \textsf{NP}-complete.
\end{conjecture}

\subsubsection{Non-orientable surfaces}
 
Let $\Sigma$ be a surface which is not necessarily orientable. Our results depend on theorems of Hass and Scott presented in Section \ref{sec:pinning_simple_multiloops_general} which apply only to orientable surfaces.

\begin{conjecture}[non-orientable surfaces]
For a non-orientable surface $\Sigma$ and $s\in \N$, what are the complexities of $\mathsf{Pin}(\Sigma, s)$ and $\mathsf{SimplePin}(\Sigma,s)$?

We conjecture that both are in \textsf{P} when $s\leq 3$ and \textsf{NP}-complete when $s\geq 4$.
\end{conjecture}

\subsubsection{SAT attacks}

The observations in Section \ref{sec:simple-multiloop} show that when $\Sigma$ is orientable there is a polynomial time algorithm which converts a loop or simple multiloop to a positive boolean formula in conjunctive normal form (\textsf{CNF}) called its mobidisc formula $\Phi(\gamma)$  (Definition \ref{def:mobidisc_cnf}), and that the pinning sets of $\gamma$ correspond to the satisfying assignments of $\Phi(\gamma)$. The problem of computing all minimal pinning sets of a multiloop is thus reduced to the problem of \emph{dualizing} $\Phi(\gamma)$ by converting it from a \textsf{CNF} to a logically equivalent positive disjunctive normal form (\textsf{DNF}).
Thus, one may ask whether \textsf{SAT}-solvers suited for such positive \textsf{CNF}s may be used to compute the minimum cardinal of their solutions more efficiently than expected. In practice, the current \textsf{SAT}-solvers for these problems are surprisingly fast (see for instance \cite{Elbassioni_dualizing_hypergraphs_2006,Murakami_Uno_Dualizing_Hypergraphs_2014}).

\begin{question}[SAT-solvers heuristics]
The positive \textsf{CNF}s arising from multiloops are not arbitrary as they come from loops in surfaces: this adds constraints given by (the genus of) the surface. Can one formulate a precise structural result in this direction?
Can one tune a \textsf{SAT}-solver to find heuristics exploiting that surface structure of these motonone \textsf{CNF}s to solve them faster?
\end{question}

\subsubsection{Random models: average parameters and complexity}

\begin{question}[average complexity]
Given a random model of multiloops in surfaces, what is the ``average'' complexity of the pinning problems as measured by the aforementioned \textsf{SAT} solvers?
\end{question}

Let us sketch a few random models of simple multiloops in a fixed or variable surface.

\paragraph{The random state model.}
Our first random model constructs, given a surface $\Sigma$ and integers $p,l\in \N_{\ge1}$, a random simple multiloop of length (in a sense that will be clear) bounded by $l$.
For simplicity assume that $\Sigma$ is orientable and add $p$ punctures to obtain a new surface $\Sigma'$ whose fundamental group is free (a sphere with $\ge 3$ punctures, a torus with $\ge 1$ puncture or any surface of genus $\ge 2$), and choose a negatively curved (say hyperbolic) metric on $\Sigma'$.
Choose a free presentation of its fundamental group, and choose a random word of length $l$ in the generators: there is a unique closed geodesic for the associated hyperbolic geodesic for any chosen hyperbolic metric on $\Sigma'$.
Finally, label the intersections and resolve them according to the following procedure so as to obtain a union of simple loops: while some strands have some self-intersections, resolve their self-intersection with the smallest label in the unique way that increases the number of strands.

\paragraph{The random spin model.}

Our second random model first constructs for given a size parameter $n\in \N$, a random surface $\Sigma$ of size $n\in \N$: choose a random $4$-valent multigraph with $n$ vertices and a random cyclic order on the vertices yielding a banded surface to which you may glue discs to obtain a closed surface.
Then construct a simple multiloop by choosing a random Euler multicycle of the embedded multigraph (for instance by taking a random walk on the edges and burning edges as they are used), and finally resolve self-intersections between strands as in the previous model.

\paragraph{The random braid model.}
Our third model constructs in any surface $\Sigma$, for any parameter $s\in \N$, a random simple multiloop with $s$ strands using the mapping class group action.
First choose a simple loop in $\Sigma$, then add several punctures in each region and apply a random element of the pure braid group of the punctured surface. Continue the process $s$ times, starting with the initial simple loop with different punctures and applying a random pure braid. Finally erase all punctures.

\subsubsection{Pinning games}

One may devise and study various multiplayer games related to pinning multiloops. In what follows we define just one and reflect on the first combinatorial game-theoretic questions that come to mind.

\begin{definition}[unpinning avoidance game]
To a multiloop $\gamma\colon \sqcup_1^s S^1\looparrowright \Sigma$ we associate the following impartial combinatorial game for $2$ players called the \emph{unpinning avoidance game}.
The initial configuration has a pin in each region; two players take turns removing pins so as to leave a pinning set after their turn; the first player with no legal move loses. Note that the states of the game are the vertices of the pinning ideal $\mathcal{PI}(\gamma)$.

We denote by $\mathsf{Unpin_{AVOID}}(\Sigma,s)$ the decision problem which  has 

\begin{itemize}[align=left, noitemsep]
    \item[Instance:] A multiloop $\gamma\colon \sqcup_{i=1}^s \Sph^1\looparrowright \Sigma$.
    \item[Question:] Does Player 1 have a winning strategy in the unpinning avoidance game for $\gamma$?
\end{itemize}
We denote the restriction of $\mathsf{Unpin_{AVOID}}(\Sigma,s)$ to simple multiloops by $\mathsf{SimpleUnpin_{AVOID}}(\Sigma,s)$.
\end{definition}

\begin{question}[nimbers]
The Sprague--Grundy theorem \cite{sprague_35,grundy1939mathematics} applies to unpinning avoidance games. Can we construct infinite families of multiloops yielding increasingly hard unpinning avoidance games (for instance whose Grundy numbers grow fast with the number of double points)? What are the Grundy numbers associated to multiloops with $s$ strands in genus $g$? How does restricting to simple multiloops affect the answers to these questions?
\end{question}

\begin{remark}[node kayles]
The unpinning avoidance game appears related (via the mobidisc formula) to a game that some authors call node kayles. Node kayles is played on a graph: Two players start with a vertex cover consisting of all vertices and take turns removing vertices from the cover so as to leave a vertex cover after each turn; the last player who can make a legal move wins. We denote by \textsf{Node Kayles} the decision problem concerned with deciding winning strategies in node kayles. It turns out that \textsf{Node Kayles} (along with other closely related games such as $\mathsf{G_{AVOID}(POS\text{ }DNF)})$ is \textsf{PSPACE}-complete \cite{schaefer_nodekayles_1976}.
\end{remark}

\begin{conjecture}[\textsf{PSPACE}-complete]
The problem $\mathsf{Unpin_{AVOID}}(\Sigma,s)$ is \textsf{PSPACE}-complete for any $s$ and $\Sigma$, and so is $\mathsf{SimpleUnpin_{AVOID}}(\Sigma,s)$ for any $s\geq 4$.
\end{conjecture}

\begin{remark}[$\mathsf{Unpin_{AVOID}}(\Sigma,s)$ versus $\mathsf{Pin}(\Sigma, s)$]
It appears that \textsf{Node Kayles} is to \textsf{Vertex Cover} as $\mathsf{Unpin_{AVOID}}(\Sigma,s)$ is to $\mathsf{Pin}(\Sigma, s)$. In particular, the complexity of $\mathsf{Pin}(\Sigma, s)$ for a given class of multiloops has limited direct bearing on the complexity of $\mathsf{Unpin_{AVOID}}(\Sigma,s)$ for that same class. One difficulty is that na\"{i}ve reductions from \textsf{Node Kayles} would require limiting mobidisc gadgets to have $2$ regions each. Another is that we do not know if \textsf{Node Kayles} remains \textsf{PSPACE}-complete in restriction to planar graphs and/or graphs with a max-degree bound. Additionally, \textsf{Node Kayles} in restriction to max-degree $3$ trees (the  complexity of which appears to be open, see for instance \cite{bodlaender_kayles_2016,kobayashi_kayles_2018}) seems likely to reduce to $\mathsf{SimpleUnpin_{AVOID}}(\Sigma,3)$, so determining the complexity of the latter is probably a difficult task for the ``meta'' reason that it would break new ground on a longstanding open question.
\end{remark}

\section{General results for simple multiloops}
\label{sec:pinning_simple_multiloops_general}

In this section, we first recall some terminology about Boolean formulae in conjunctive normal form that are positive.
This will serve to set up a correspondence between the pinning sets of loops or simple multiloops and the satisfying assignments of such a formula.
After reviewing the \textsf{Vertex Cover} problem and a few of its variants, we use these ideas to prove a preliminary result (Theorem \ref{thm:simple-no-strand-restiction-hard}) that pinning simple multiloops with arbitrary many strands is \textsf{NP}-complete.

\subsection{Boolean formulae in conjunctive normal form}

\begin{definition}[formulae and satisfying assignments]
Consider a set of variables $x=\{x_1,\dots,x_r\}$ which may take boolean values $\{\true,\false\}$.
A \emph{formula} $\phi$ is an expression in the variables $x$ involving negations $\neg$, disjunctions $\vee$ , and conjunctions $\wedge$ (and parentheses). 

Evaluating the variables $x$ at $\xi \in \{\true,\false\}^r$ yields a value $\phi(\xi)\in \{\true,\false\}$ for the formula $\phi$.
Such an assignment $\xi$ of the variables $x$ is said to \emph{satisfy} $\phi$ when $\phi(\xi)=\true$.

Two formulae $\phi,\phi'$ on the set of variables $x$ are called \emph{logically equivalent} when their evaluations on every assignment $\xi$ of the variables $x$ are equal $\phi(\xi)=\phi'(\xi)$. 
\end{definition}

It is a classical observation that every formula is equivalent to a conjunctive normal form, which we now define.

\begin{definition}[positive conjunctive normal forms]
Consider a set of boolean variables $x=\{x_1,\dots,x_r\}$.

A \emph{clause} is a formula of the form $l_1\vee \dots \vee l_k$ where each $l_j$ is a \emph{literal}, namely a variable $x_t$ or its negation $\neg x_i$.
A \emph{conjunctive normal form} (abbreviated \textsf{CNF}) is a formula which is conjunction $c_1\wedge \dots\wedge c_m$ of clauses $c_j$.
A conjunctive normal form is called \emph{positive} (or \emph{monotone}) when none of the clauses feature negated variables.
\end{definition}

\begin{definition}[solution ideal]
    Consider a positive conjunctive normal form $\phi$ over a set of boolean variables $x=\{x_1,\dots,x_r\}$.
    
    For an assignment $\xi$ of $x$ satisfying $\phi$, we call $\{x_i \colon \xi_i=\true\}\in \Parts(x)$ a \emph{solution} of $\phi$. 
    
    The solutions of $\phi$ form an ideal $\mathcal{S}(\phi) \subset \Parts(x)$: a sub-poset which is absorbing under union, and it contains the whole set of all elements $\{x_1,\dots,x_r\}$.
\end{definition}

\begin{remark}[pruning]
    \label{rem:prune}
    By definition if two positive \textsf{CNF} are equivalent then their solution ideals are isomorphic.
    Thus to compute the solution ideal, we may first ``prune'' it by retaining only those \emph{innermost} clauses which are implied by no other clause.
\end{remark}

\begin{remark}[hypergraph vertex cover]
    The solutions of a positive conjunctive normal form correspond to the vertex covers of the hypergraph whose vertices are indexed by the variables and hyperedges correspond to the clauses.
\end{remark}

\subsection{Pinning simple multiloops is \textsf{NP}}
\label{sec:simple-multiloop}

Throughout this subsection, we fix an oriented surface $\Sigma$. 

Before dealing with simple multiloops, let first recall from \cite[Section 2]{simon_stucky_2025pinningidealmultiloop} how the problem of pinning loops reduces to a problem regarding satisfying assignments of its mobidisc formula, a positive \textsf{CNF}.

Hass and Scott \cite{Hass-Scott_Intersection-curves-surfaces_1985} characterize taut loops in terms of the absence of singular monogons and bigons, which we define below for multiloops and subsequently use for simple multiloops.

\begin{definition}[singular, embedded, and regional monorbigons]
\label{def:singular-monorbigon}

Consider a multiloop $\gamma \colon \sqcup_{i=1}^s\Sph^1\looparrowright \Sigma$. 
A strand of $\gamma$ obtained by restriction to one of the circles will be denoted $\gamma_i\colon \Sph^1_i\looparrowright \Sigma$.

A \emph{singular monogon} of the strand $\gamma_i$ is a non-trivial closed interval $I\subset \Sph^1_i$ such that $\gamma(\partial I)=\{x\}$ is a double point of $\gamma$ and $\gamma(I)$ is null-homotopic.
This singular monogon is \emph{embedded} when $\gamma$ is injective on the interior of $I$, and \emph{regional} when $\gamma(I)$ bounds a region of $\gamma$.

A \emph{singular bigon} between strands $\gamma_i,\gamma_j$ (which may coincide) is a disjoint union of non-trivial closed intervals $I \sqcup J$ where $I\subset \Sph^1_i$ and $J\subset \Sph^1_j$ such that $\gamma(\partial I)=\{x,y\}=\gamma(\partial J)$ for distinct double points $x,y$ of $\gamma$ and $\gamma(I\sqcup J) \subset \Sigma$ is null-homotopic.
This singular bigon is \emph{embedded} when $\gamma$ is injective on the interior of $I\sqcup J$, and \emph{regional} when $\gamma(I\sqcup J)$ bounds a region of $\gamma$.

A \emph{singular monorbigon} $K$ refers to a singular monogon $I$ or singular bigon $I\sqcup J$.
The restriction $\gamma(K)$ is a subloop of $\gamma$ and we call $\gamma(\partial K)$ its \emph{marked points}.
\end{definition}

\begin{figure}[H]
    \centering
    \scalebox{0.26}{\scalefont{2}{\input{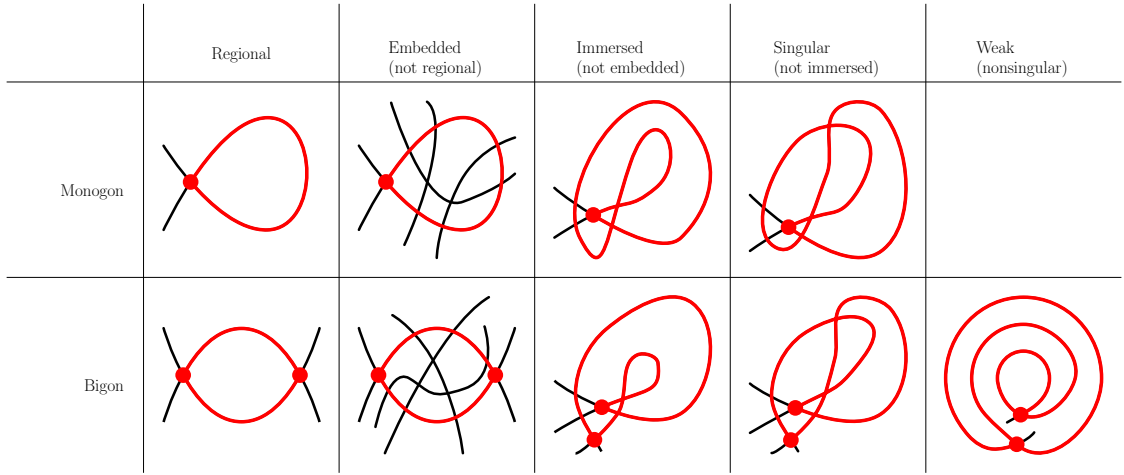}}}
    \caption{Examples and non-examples of different kinds of monorbigons.}
    \label{fig:monorbigons}
\end{figure}

\begin{definition}[immersed monorbigons and mobidiscs]
\label{def:immersed-monorbigon-mobidiscs}
    For a multiloop $\gamma \colon \sqcup_{i=1}^s\Sph^1\looparrowright \Sigma$, a singular monorbigon $K$ of $\gamma$ is called \emph{immersed} when there exists an immersion $\iota\colon \Disc\looparrowright \Sigma$ such that the restriction $\gamma \colon K \looparrowright \Sigma$ factors through $\iota \colon \partial \Disc \looparrowright \Sigma$.
    The image $\iota(\Disc)$ of such an immersion covers a set of regions of $\gamma$, thus defining an element in $\mathcal{P}(\mathcal{R}(\gamma))$, called a \emph{mobidisc}. 
    The set of all mobidiscs $\MoB(\gamma)\subset \mathcal{P}(\mathcal{R})$ defines the hyperedges of a hypergraph with vertices $\mathcal{R}$. 
    
    An immersed monorbigon or its associated mobidisc is called \emph{innermost} when the associated mobidisc contains no other mobidiscs of $\gamma$. In particular, a regional monorbigon gives rise to an innermost mobidisc with exactly one region.
\end{definition}

\begin{remark}[embedded versus regional]
Figure \ref{fig:monorbigons} illustrates monorbigons with increasing restrictions from (non-singular to) singular to immersed to embedded to regional. Most important to this work is the difference between embedded and regional bigons.
\end{remark}

The following is implicit in \cite{Hass-Scott_Intersection-curves-surfaces_1985} as explained in \cite{simon_stucky_2025pinningidealmultiloop}. It remains true with ``mobidisc'' replaced with ``innermost mobidisc''.

\begin{theorem}[pinning loops with mobidiscs]
    \label{thm:pinning-mobidiscs}
    For a loop $\gamma \colon \Sph^1 \looparrowright \Sigma$, a subset of regions $P\subset\mathcal{P}(\mathcal{R}(\gamma))$ is pinning $\gamma$ if and only if every mobidisc $D\in \MoB(\gamma)$ has $D\cap P\ne \emptyset$.
\end{theorem}

\begin{definition}[mobidisc formula]
\label{def:mobidisc_cnf}
To a multiloop $\gamma \colon \sqcup_{i=1}^s\Sph^1  \looparrowright \Sigma$ we associate its \emph{mobidisc formula} denoted by $\Phi(\gamma)$: it is the positive boolean formula in conjunctive normal form on the set of boolean variables indexed by $\mathcal{R}(\gamma)$ whose clauses correspond to $\MoB(\gamma)\subset \mathcal{P}(\mathcal{R})$.
\end{definition}

The following is a reformulation of Theorem \ref{thm:pinning-mobidiscs}.

\begin{corollary}[pinning a loop = satisfying its mobidisc formula]
\label{cor:LooPin-to-mobidisc-CNF}
The pinning sets of a loop correspond to the solutions of its mobidisc formula.
In particular, the pinning ideal of a loop $\gamma$ is isomorphic to the solution ideal of $\Phi(\gamma)$. 
\end{corollary}

In this work we are interested in multiloops, however the following Remark \ref{rem:one-strand-only} shows that Corollary \ref{cor:LooPin-to-mobidisc-CNF} does not generalize so immediately.

\begin{remark}[one strand only]
    \label{rem:one-strand-only}
    Theorem \ref{thm:pinning-mobidiscs} is false for (non-simple) multiloops with more than one strand. Indeed the example below is based on \cite[Figure 0.1]{Hass-Scott_Intersection-curves-surfaces_1985} and exhibits a multiloop $\Sph^1 \sqcup \Sph^1 \looparrowright \Sph^2\setminus \{p_1,p_2\}$ which is not taut and has no singular monorbigons.
\end{remark}

 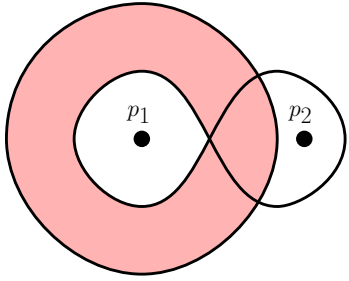
\begin{figure}[H]
        \centering
        \scalebox{0.36}{\definecolor{c010000}{RGB}{1,0,0}

\def \globalscale {1.000000}
\begin{tikzpicture}[y=1cm, x=1cm, yscale=\globalscale,xscale=\globalscale, every node/.append style={scale=\globalscale}, inner sep=0pt, outer sep=0pt]
  %path54
  \path[draw=red,fill=red,fill opacity=0.3,line cap=butt,line join=miter,line width=0.0265cm,miter limit=4.0] (5.0265, 10.1852).. controls (2.5265, 10.1852) and (0.0265, 8.1852) .. (0.0265, 5.1852).. controls (0.0265, 2.6852) and (2.5265, 0.1852) .. (5.0265, 0.1852).. controls (7.5265, 0.1852) and (10.0265, 2.6852) .. (10.0265, 5.1852).. controls (10.0265, 7.6852) and (7.5265, 10.1852) .. (5.0265, 10.1852) -- cycle(5.0265, 7.6852).. controls (6.205, 7.6852) and (6.8657, 6.4352) .. (7.5265, 5.1852).. controls (6.8657, 3.9352) and (6.205, 2.6852) .. (5.0265, 2.6852).. controls (3.848, 2.6852) and (2.5265, 4.0067) .. (2.5265, 5.1852).. controls (2.5265, 6.3637) and (3.848, 7.6852) .. (5.0265, 7.6852) -- cycle;

  %path41
  \path[draw=black,line cap=butt,line join=miter,line width=0.1058cm] (5.0265, 10.1852).. controls (2.5265, 10.1852) and (0.0265, 8.1852) .. (0.0265, 5.1852).. controls (0.0265, 2.6852) and (2.5265, 0.1852) .. (5.0265, 0.1852).. controls (7.5265, 0.1852) and (10.0265, 2.6852) .. (10.0265, 5.1852).. controls (10.0265, 7.6852) and (7.5265, 10.1852) .. (5.0265, 10.1852) -- cycle;

  %path42
  \path[draw=black,line cap=butt,line join=miter,line width=0.1058cm] (5.0265, 7.6852).. controls (7.3835, 7.6852) and (7.6694, 2.6852) .. (10.0265, 2.6852).. controls (10.8605, 2.6852) and (12.5398, 3.7133) .. (12.5332, 5.1852).. controls (12.5265, 6.6707) and (10.8605, 7.6852) .. (10.0265, 7.6852).. controls (7.6694, 7.6852) and (7.3835, 2.6852) .. (5.0265, 2.6852).. controls (3.8479, 2.6852) and (2.5265, 4.0067) .. (2.5265, 5.1852).. controls (2.5265, 6.3637) and (3.8479, 7.6852) .. (5.0265, 7.6852) -- cycle;

  %path43
  \path[fill=c010000,even odd rule,line cap=round,line width=0.0873cm] (5.0265, 5.1852) circle (0.3044cm);

  %circle43
  \path[fill=c010000,even odd rule,line cap=round,line width=0.0873cm] (11.0265, 5.1852) circle (0.3044cm);

  \tikzstyle{every node}=[font=\fontsize{35}{35}\selectfont]
%text54
  \node[anchor=south west] (text54) at (4.4934, 5.778){$p_1$};

  %text57
  \node[anchor=south west] (text57) at (10.473, 5.778){$p_2$};

\end{tikzpicture}}
        \caption{This multiloop $\Sph^1 \sqcup \Sph^1 \looparrowright \Sph^2\setminus \{p_1,p_2\}$ is not taut yet has no singular monorbigons.}        \label{fig:no_singular_monorbigon_non_taut_multiloop}
\end{figure}

For simple multiloops, we obtain an analog to Theorem \ref{thm:pinning-mobidiscs} as a corollary of the following. It remains true with ``embedded bigon'' replaced with ``innermost embedded bigon''.

\begin{lemma}[embedded bigons in non-taut simple multiloops] 
\label{lem:simple_multiloops_have_embedded_monorbigons}
If $\gamma\colon \sqcup_{i=1}^s \Sph^1\looparrowright \Sigma$ is a simple multiloop which is not taut, then it has an embedded bigon.
\end{lemma}

\begin{proof}
    This follows from \cite{Hass-Scott_Intersection-curves-surfaces_1985}[Lemma 3.1].
\end{proof}

\begin{corollary}[pinning simple multiloops with embedded bigons]
    \label{cor:pinning-simple-mobidiscs}
    
    For a simple multiloop $\gamma \colon \sqcup_{i=1}^s\Sph^1\looparrowright \Sigma$, a subset of regions $P\in\mathcal{P}(\mathcal{R}(\gamma))$ is pinning $\gamma$ if and only if every mobidisc $D\in \MoB(\gamma)$ has $D\cap P\ne \emptyset$. 
    Hence, the pinning ideal of $\gamma$ is isomorphic to the solution ideal of its mobidisc formula  $\Phi(\gamma)$.
\end{corollary}

\begin{remark}[pruning $\Phi(\gamma)$ is equivalent to passing to innermost mobidiscs] By Remark \ref{rem:prune}, whenever convenient, we need only consider innermost mobidiscs when computing $\Phi(\gamma)$.    
\end{remark}

Since the mobidisc formula may be computed in polynomial time for simple multiloops by a na\"{i}ve cubic algorithm, we may verify pinning sets in polynomial time.

\begin{corollary}[pinning simple multiloops is \textsf{NP}]
    \label{cor:pinning-simple-easy}
    For a fixed orientable surface $\Sigma$ we have:
    \begin{enumerate}[noitemsep]
        \item For every $s\in \N$, the problem $\mathsf{SimplePin}(\Sigma,s)$ is in \textsf{NP}.
        \item The problem $\mathsf{SimplePin}(\Sigma)$ is in \textsf{NP}.
    \end{enumerate}
\end{corollary}

\subsection{Pinning simple multiloops is \textsf{NP}-complete}

This subsection first defines the \textsf{Vertex Cover} problem and the variants that we will need for the purposes of this work. 
Then a preliminary result is proved: the complexity of pinning simple multiloops with no restriction on the number of strands is \textsf{NP}-complete.

\begin{definition}[\textsf{Vertex Cover} and variants]
\label{def:vc-prob}
The \textsf{Vertex Cover (VC)} problem has:
\begin{itemize}[align=left, noitemsep]
    \item[Instance:] A simple graph $G=(V,E)$ and an integer $k$.
    \item[Question:] Is there $U\subset V$ with $\Card(U)\leq k$ such that for all $\{v_1,v_2\}\in E$, at least one of $v_1$ and $v_2$ belongs to $U$?
\end{itemize}
We also define the following variants of \textsf{VC}:
\begin{itemize}
    \item \textsf{Planar Vertex Cover (PVC)} is the restriction of \textsf{VC} to input graphs which are planar.
    \item \textsf{3-Connected Cubic Planar Vertex Cover (\textsf{3C3PVC})} is the restriction of \textsf{PVC} to $3$-connected cubic graphs. Here, $3$-connected means that the graph is connected upon removal of any set of at most $2$ vertices, and cubic means that every vertex has degree $3$.
    \item \textsf{Bipartite Vertex Cover (BVC)} is the restriction of \textsf{VC} to input graphs which are bipartite.
\end{itemize}
\end{definition}

\begin{lemma}[\textsf{NP}-complete \textsf{Vertex Cover} variants] The problems \textsf{VC} and \textsf{PVC} are \textsf{NP}-complete.
\label{lem:vc_hardness}
\end{lemma}

\begin{proof}
It is well-known (and straightforward) that \textsf{VC} is in \textsf{NP}, and thus so is the restriction \textsf{PVC}. It is also classical that \textsf{VC} is \textsf{NP}-hard \cite{Karp1972}, and that \textsf{PVC} is \textsf{NP}-hard \cite{Garey-Johnson-Stockmeyer_simplified-NP-complete-graph_1976}[Theorem 2.7].
\end{proof}

We will show that \textsf{BVC} is in \textsf{P} in Lemma \ref{lem:bvc_easy}, and that \textsf{3C3PVC} is \textsf{NP}-complete in Lemma \ref{lem:pvc3c_hard}.

\begin{theorem}[pinning simple multiloops with no strand restriction is \textsf{NP}-hard]
    \label{thm:pinning-simple-hard}
    For a fixed orientable surface $\Sigma$, the problem $\mathsf{SimplePin}(\Sigma)$ is \textsf{NP}-hard.
\label{thm:simple-no-strand-restiction-hard}
\end{theorem}

\begin{proof}
In light of Remark \ref{rmk:plane_enough}, it is enough to prove this when $\Sigma=\R^2$.

One may show this using a simplified adaption of the reduction from \textsf{PVC} described in \cite[Section 2.2]{simon_stucky_2025pinningidealmultiloop}. It also follows from the proof of Theorem \ref{thm:20-hard} that we will give in Section \ref{sec:20hard}.

Let us sketch yet another reduction from \textsf{3C3PVC} (depending on Lemma \ref{lem:pvc3c_hard}) which we find to be elegant and which is similar to constructions in the next section. 
Consider an instance $(G,k)$ of \textsf{3C3PVC}.
Note that since $G$ is $3$-connected, it is \emph{bridgeless} ($2$-edge connected).

Embed $G$ in $\R^2$ in linear time with straight-line edges (using for instance \cite{Hopcroft-Tarjan_planar-embedding_1974}). 
Since $G$ is bridgeless, every edge bounds different faces on either side, hence the boundary of each face defines a cycle of edges and vertices in which each one appears exactly once. To each face of $G$ associate a loop surrounding only its edges and vertices, so that their union is a simple multiloop whose regions correspond to the vertices and faces of $G$. The innermost bigons correspond to the edges of $G$.
This defines a polynomial time reduction from an instance of $\mathsf{3C3PVC}$ to $\mathsf{SimplePin}(\Sigma)$. See Figure \ref{fig:pentaprism} for an example.

\begin{figure}[H]
    \centering
    \scalebox{.6}{\input{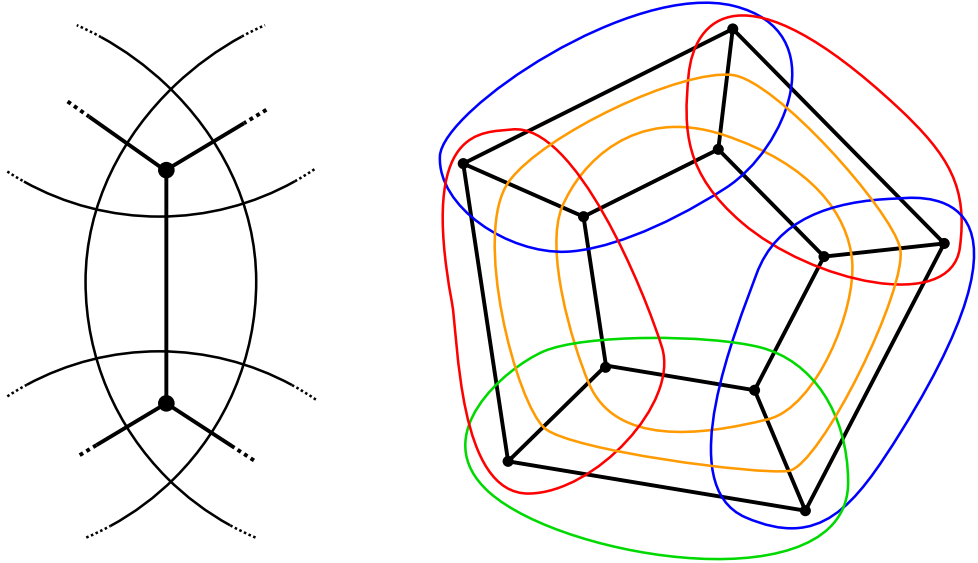}}
    \caption{Left: Local strands around an edge. Right: A $4$-colored multiloop associated to the $1$-skeleton of the pentagonal prism.}
    \label{fig:pentaprism}
\end{figure}

We must now show that $\gamma$ can be pinned with at most $k$ pins if an only $G$ has a vertex cover of size at most $k$.
This will follow from Corollary \ref{cor:pinning-simple-mobidiscs}, after making some observations about the mobidisc formula and its optimal satisfying assigements.
Every innermost bigon of $\gamma$ consist of exactly three regions (two triangular regions corresponding to vertices of $G$ and one quadrangular regions corresponding to the edge between them), hence the mobidiscs of $\gamma$ are in one-to-one correspondence with the edges of $G$. Moreover, any two mobidiscs of $\gamma$ intersect in at most one triangular region corresponding to the vertex of $G$ in that region. Hence the problem of choosing which of the two vertices adjacent to an edge of $G$ belongs to a vertex cover is equivalent to the problem of choosing which of the two triangular regions of the associated mobidisc of $\gamma$ receives a pin (note that optimal pinning sets avoid the quadrangular regions of each mobidisc). 
\end{proof}

The proof of Theorem \ref{thm:simple-no-strand-restiction-hard} reveals an interesting connection to the $4$-color theorem \cite{appel_haken_1976_4color}, to which we will refer again in Section \ref{sec:4-strands} in relation to Conjecture \ref{conj:4-hard}.

\begin{remark}[$4$-color theorem]
\label{rmk:4color}
The construction in the proof of Theorem \ref{thm:simple-no-strand-restiction-hard} gives a way to $4$-color the strands of a $\gamma$ in such a way that no two strands of the same color intersect each other. Indeed, after $4$-coloring the faces of $G$ we may color each strand of $\gamma$ according to the color of the face it surrounds. In fact this may be done in quadratic time \cite{robertson_et_al_1996_efficient4color}.
\end{remark}

We now turn our attention to the complexity of $\mathsf{SimplePin}(\Sigma,s)$ for fixed $s$.

\section{Simple multiloops with few strands}
\label{sec:3easy}

In this section, we show that $\mathsf{SimplePin}(\Sigma,3)$ is in \textsf{P} using a reduction to the \textsf{Bipartite Vertex Cover (BVC)} problem (Definition \ref{def:vc-prob}) which is in \textsf{P}.

\subsection{The \textsf{Bipartite Vertex Cover} problem}

We have the following well-known lemma, whose proof we outline for completeness.

\begin{lemma}[\textsf{BVC} is in \textsf{P}]
\label{lem:bvc_easy}
    The problem \textsf{BVC} is in \textsf{P}.
\end{lemma}

\begin{proof}
    Let $(G,k)$ be a problem instance of \textsf{BVC}. Since $G$ is bipartite, we may use Dinitz' algorithm \cite{dinitz_1970,edmonds_karp_72} to find a maximum matching (a matching of maximum cardinality) of $G$ in polynomial time. Bondy and Murty describe in \cite{bondy1976graph}[Theorem 5.3, pages 74-75] a constructive proof of K\fH{o}nig's theorem converting this matching to a minimum vertex cover (a vertex cover of minimum cardinality) in polynomial time. Denoting by $k'$ the size of this minimum vertex cover, the graph $G$ has vertex cover of size $\le k$ if and only if $k'\leq k$.
\end{proof}

\subsection{Pinning $3$-strand simple multiloops is in \textsf{P}}
\label{subsec:3easy}

We now prove Theorem \ref{thm:3-easy}, which we restate below.

\begin{theorem}
For every orientable surface $\Sigma$ and $s\leq 3$, the problem $\mathsf{SimplePin}(\Sigma,s)$ is in \textsf{P}.
\end{theorem}

\begin{proof}
Let $(\gamma,k)$ be an instance of $\mathsf{SimplePin}(\Sigma,s)$ with $s\leq 3$. We may assume $s=3$ by adding an appropriate number of embedded circles disjoint from the rest of $\gamma$. We will construct a bipartite graph in polynomial time which has a vertex cover of size at most $k'$ (defined below and satisfying $k'\leq k$) if and only if $\gamma$ admits a pinning set of size $\le k$, and appeal to Lemma \ref{lem:bvc_easy} to complete the proof. 

By Corollary \ref{cor:pinning-simple-mobidiscs}, the simple multiloop $\gamma$ is pinned if and only if every one of its innermost bigons receives a pin.
Note that each innermost bigon $B$ of $\gamma$ comprises two arcs of strands $\gamma_1$ and $\gamma_2$ with exactly two intersection points marked as $x$ and $y$. Only arcs of the third strand $\gamma_3$ may run through $B$. Since $B$ is innermost, such arcs of $\gamma_3$ will cross $B$ from an intersection with $\gamma_1$ to an intersection with $\gamma_2$. Since $\gamma_3$ is simple, its arcs must divide $B$ into a sequence of regions $T_x, R_1, R_2, \ldots, R_n, T_y$ for some $n\in \N$ where $T_x$ and $T_y$ are triangular regions and every $R_i$ is a quadrangular region which cannot belong to another innermost bigon. Figure \ref{fig:bigon} depicts such a bigon in $\gamma$ bounded by arcs of $\gamma_1$ and $\gamma_2$ and its intersection with $\gamma_3$.

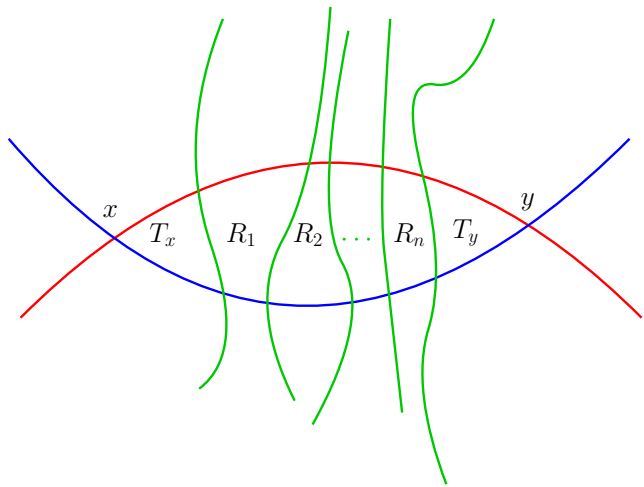
\begin{figure}[H]
    \centering
    \scalebox{.6}{\definecolor{c00c800}{RGB}{0,200,0}

\def \globalscale {1.000000}
\begin{tikzpicture}[y=1cm, x=1cm, yscale=\globalscale,xscale=\globalscale, every node/.append style={scale=\globalscale}, inner sep=0pt, outer sep=0pt]
  %path1
  \path[draw=red,line cap=butt,line join=miter,line width=0.0529cm,miter limit=10.0] (0.3043, 3.7439) .. controls (4.8904, 8.33) and (9.4765, 8.33) .. (14.0626, 3.7439);

  %path2
  \path[draw=blue,line cap=butt,line join=miter,line width=0.0529cm,miter limit=10.0] (0.0397, 7.7126) .. controls (4.273, 2.7737) and (8.8591, 2.7737) .. (13.798, 7.7126);

  %path4
  \path[draw=c00c800,line cap=butt,line join=miter,line width=0.0529cm,miter limit=10.0] (4.2598, 2.1696) .. controls (4.9653, 2.6987) and (5.0535, 3.7571) .. (4.5244, 5.3446) .. controls (3.9952, 6.9321) and (4.0834, 8.6078) .. (4.789, 10.3717);

  %path6
  \path[draw=c00c800,line cap=butt,line join=miter,line width=0.0529cm,miter limit=10.0] (6.3765, 1.905) .. controls (5.6709, 3.3161) and (5.5827, 4.5067) .. (6.1119, 5.4769) .. controls (6.641, 6.447) and (6.9938, 8.1668) .. (7.1702, 10.6363);

  %path8
  \path[draw=c00c800,line cap=butt,line join=miter,line width=0.0529cm,miter limit=10.0] (6.7733, 1.3758) .. controls (7.6553, 2.9633) and (7.8758, 4.154) .. (7.4348, 4.9477) .. controls (6.9938, 5.7415) and (7.0379, 7.4612) .. (7.5671, 10.1071);

  %path10
  \path[draw=c00c800,line cap=butt,line join=miter,line width=0.0529cm,miter limit=10.0] (9.7367, 0.0529) .. controls (9.2075, 1.464) and (9.0752, 2.6106) .. (9.3398, 3.4925) .. controls (9.6044, 4.3744) and (9.5603, 5.521) .. (9.2075, 6.9321) .. controls (8.8547, 8.3432) and (8.9429, 9.0047) .. (9.4721, 8.9165) .. controls (10.0013, 8.8283) and (10.4422, 9.3133) .. (10.795, 10.3717);

  %path28
  \path[draw=c00c800,line cap=butt,line join=miter,line width=0.0529cm,miter limit=10.0] (8.7577, 1.6404) .. controls (8.5813, 3.2279) and (8.449, 4.4185) .. (8.3608, 5.2123) .. controls (8.2726, 6.006) and (8.3167, 7.7258) .. (8.4931, 10.3717);

  \tikzstyle{every node}=[font=\fontsize{18}{18}\selectfont]
%text41
  \node[anchor=south west] (text41) at (3.1621, 5.3346){$T_x$};

  %text42
  \node[anchor=south west] (text42) at (4.8554, 5.3346){$R_1$};

  %text43
  \node[anchor=south west] (text43) at (6.3371, 5.3346){$R_2$};

  %text44
  \node[text=c00c800,anchor=south west] (text44) at (7.3954, 5.3346){$\cdots$};

  %text45
  \node[anchor=south west] (text45) at (8.5596, 5.3346){$R_n$};

  %text46
  \node[anchor=south west] (text46) at (9.8825, 5.3346){$T_y$};

  %text47
  \node[anchor=south west] (text47) at (2.1285, 5.9719){$x$};

  %text48
  \node[anchor=south west] (text48) at (11.3889, 6.1307){$y$};

\end{tikzpicture}}
    \caption{The structure of innermost bigons in $3$-strand simple multiloops}
    \label{fig:bigon}
\end{figure}

We now construct a graph $G$ as follows. It has two vertices and one edge for each innermost (therefore embedded) bigon $B$ of $\gamma$ which is not regional (namely with at least one arc of the third strand passing through it). The vertices of the edge are the triangular regions $T_x$ and $T_y$. Since every triangular region of $\gamma$ belongs to at most $3$ innermost bigons, the vertices have degree $1$, $2$, or $3$. See Figure \ref{fig:loop} for an example of this construction.

\begin{figure}[h]
    \centering
    \scalebox{.6}{\definecolor{c00c800}{RGB}{0,200,0}

\def \globalscale {1.000000}
\begin{tikzpicture}[y=1cm, x=1cm, yscale=\globalscale,xscale=\globalscale, every node/.append style={scale=\globalscale}, inner sep=0pt, outer sep=0pt]
  %rect2
  \path[draw=c00c800,line width=0.0794cm,rounded corners=0.7938cm] (3.1882, 9.0355) rectangle (8.4799, 0.0397);

  %rect3
  \path[draw=red,line width=0.0794cm,rounded corners=0.5953cm] (0.2778, 6.1251) rectangle (14.3007, 2.1564);

  %rect5
  \path[draw=blue,line width=0.0794cm,rounded corners=0.5953cm] (7.157, 7.448) rectangle (11.1257, 1.098);

  %path51
  \path[draw=black,miter limit=10.0] (18.0049, 6.6543) -- (21.9736, 6.6543);

  %ellipse53
  \path[draw=black,fill=white] (17.608, 6.6543) circle (0.3969cm);

  %path57
  \path[draw=black,miter limit=10.0] (22.7674, 6.6543) -- (24.3549, 6.6543);

  %ellipse59
  \path[draw=black,fill=white] (22.3705, 6.6543) circle (0.3969cm);

  %path63
  \path[draw=black,miter limit=10.0] (24.7518, 6.2574) -- (24.7518, 2.8178);

  %ellipse65
  \path[draw=black,fill=white] (24.7518, 6.6543) circle (0.3969cm);

  %path69
  \path[draw=black,miter limit=10.0] (17.608, 2.8178) -- (17.608, 6.2574);

  %path71
  \path[draw=black,miter limit=10.0] (18.0049, 2.4209) -- (21.9736, 2.4209);

  %ellipse73
  \path[draw=black,fill=white] (17.608, 2.4209) circle (0.3969cm);

  %path77
  \path[draw=black,miter limit=10.0] (22.3705, 2.8178) -- (22.3705, 6.2574);

  %ellipse79
  \path[draw=black,fill=white] (22.3705, 2.4209) circle (0.3969cm);

  %path83
  \path[draw=black,miter limit=10.0] (24.3549, 2.4209) -- (22.7674, 2.4209);

  %ellipse85
  \path[draw=black,fill=white] (24.7518, 2.4209) circle (0.3969cm);

  \tikzstyle{every node}=[font=\fontsize{18}{18}\selectfont]
%text92
  \node[anchor=south west] (text92) at (3.79, 8.0974){$1$};

  %text93
  \node[anchor=south west] (text93) at (3.79, 5.3193){$4$};

  %text94
  \node[anchor=south west] (text94) at (0.6679, 5.3193){$3$};

  %text95
  \node[anchor=south west] (text95) at (7.5471, 5.3193){$5$};

  %text96
  \node[anchor=south west] (text96) at (7.5867, 6.629){$6$};

  %text97
  \node[anchor=south west] (text97) at (8.7377, 6.625){$2$};

  %text98
  \node[anchor=south west] (text98) at (8.7509, 5.3159){$7$};

  %text99
  \node[anchor=south west] (text99) at (11.5158, 5.3193){$8$};

  %text100
  \node[anchor=south west] (text100) at (8.7578, 1.5093){$11$};

  %text101
  \node[anchor=south west] (text101) at (7.4614, 1.5093){$10$};

  %text102
  \node[anchor=south west] (text102) at (3.79, 1.5093){$9$};

  %text103
  \node[anchor=south west] (text103) at (17.4822, 2.2634){$9$};

  %text104
  \node[anchor=south west] (text104) at (17.4689, 6.457){$1$};

  %text105
  \node[anchor=south west] (text105) at (22.2314, 6.457){$6$};

  %text106
  \node[anchor=south west] (text106) at (24.6127, 6.457){$2$};

  %text107
  \node[anchor=south west] (text107) at (22.0796, 2.2634){$10$};

  %text108
  \node[anchor=south west] (text108) at (24.4609, 2.2634){$11$};

\end{tikzpicture}}
    \caption{Example of a $3$-strand simple multiloop and the graph whose vertex covers correspond to pinning sets. Its mobidisc formula is $3 \land 8 \land (9 \vee 10) \land (2 \vee 6) \land (10 \vee 11) \land (1\vee 6) \land (1\vee 4 \vee 9) \land (6\vee 5 \vee 10) \land (2\vee 7 \vee 11)$.}
    \label{fig:loop}
\end{figure}
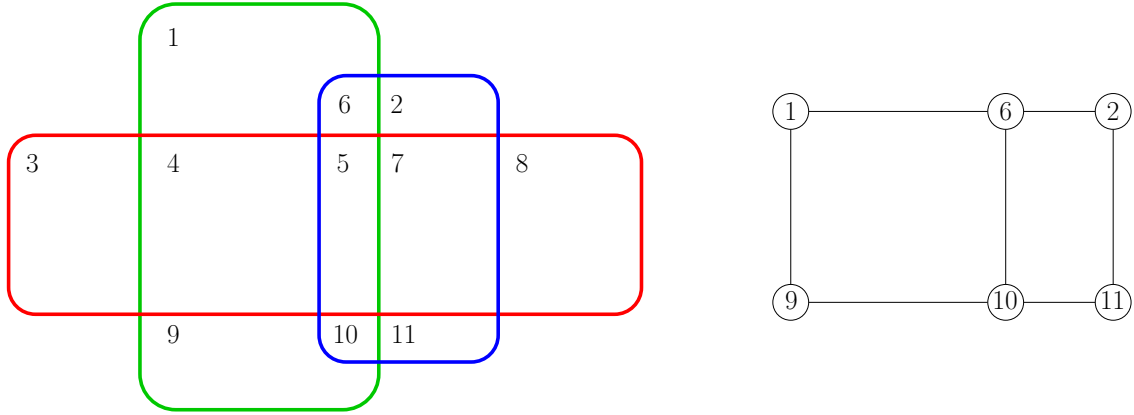

Let us now show that $G$ is bipartite by constructing an orientation of its edges such that every vertex is either a source or sink: the sources and sinks will be independent sets.
For this, we first color the $3$ strands $\gamma_r,\gamma_g,\gamma_b$ by the colors $\{r, g, b\}$. Next, color each edge of $G$ by one of $\{gb, br, rg\}$ according to the colors of the arcs bounding the associated bigon.
Finally, orient the edges of $G$ using the orientation of $\Sigma$ and its coloring so that traveling along it places the first color in its coloring on the left of the traveler. See Figure \ref{fig:orient} to visualize this construction.

\begin{figure}[H]
    \centering
    \scalebox{.6}{\definecolor{lime}{RGB}{0,255,0}
\definecolor{c00c800}{RGB}{0,200,0}

\def \globalscale {1.000000}
\begin{tikzpicture}[y=1cm, x=1cm, yscale=\globalscale,xscale=\globalscale, every node/.append style={scale=\globalscale}, inner sep=0pt, outer sep=0pt]
  %path1
  \path[draw=black,miter limit=10.0] (4.8583, 6.4193) -- (7.9211, 9.4821);

  %ellipse1
  \path[draw=black,fill=black] (4.8022, 6.3632) circle (0.0794cm);

  %ellipse2
  \path[draw=black,fill=black] (7.9772, 9.5382) circle (0.0794cm);

  %path3
  \path[draw=black,miter limit=10.0] (1.6833, 9.4821) -- (4.7461, 6.4193);

  %ellipse3
  \path[draw=black,fill=black] (1.6272, 9.5382) circle (0.0794cm);

  %ellipse4
  \path[draw=black,fill=black] (4.8022, 6.3632) circle (0.0794cm);

  %path5
  \path[draw=black,miter limit=10.0] (4.8022, 2.7384) -- (4.8022, 6.2839);

  %ellipse5
  \path[draw=black,fill=black] (4.8022, 2.6591) circle (0.0794cm);

  %ellipse6
  \path[draw=black,fill=black] (4.8022, 6.3632) circle (0.0794cm);

  %path7
  \path[draw=blue,line width=0.0529cm,miter limit=10.0] (6.1251, 0.5424) .. controls (4.714, 2.1299) and (3.9202, 3.4087) .. (3.7439, 4.3789) .. controls (3.5675, 5.349) and (3.5234, 6.1427) .. (3.6116, 6.7601) .. controls (3.6998, 7.3775) and (4.0084, 7.9507) .. (4.5376, 8.4799) .. controls (5.0668, 9.0091) and (5.7282, 9.4059) .. (6.522, 9.6705) .. controls (7.3157, 9.9351) and (8.9473, 9.891) .. (11.4168, 9.5382);

  %path9
  \path[draw=lime,line width=0.0529cm,miter limit=10.0] (4.273, 0.807) .. controls (5.5077, 1.8653) and (5.9928, 3.541) .. (5.7282, 5.8341) .. controls (5.4636, 8.1271) and (3.5675, 9.7146) .. (0.0397, 10.5966);

  %path11
  \path[draw=red,line width=0.0529cm,miter limit=10.0] (0.8334, 11.3903) .. controls (0.8334, 8.7445) and (1.4949, 7.0247) .. (2.8178, 6.2309) .. controls (4.1407, 5.4372) and (5.5077, 5.4813) .. (6.9189, 6.3632) .. controls (8.33, 7.2452) and (9.3001, 8.7445) .. (9.8293, 10.8611);

  %path13
  \path[draw=black,miter limit=10.0] (5.0668, 6.6278) -- (6.2145, 7.7756);

  %path14
  \path[draw=black,fill=black,miter limit=10.0] (6.3688, 7.9298) -- (6.2659, 7.6211) -- (6.2145, 7.7756) -- (6.06, 7.8269) -- cycle;

  %path15
  \path[draw=black,miter limit=10.0] (4.8022, 5.8341) -- (4.8022, 4.4945);

  %path16
  \path[draw=black,fill=black,miter limit=10.0] (4.8022, 4.2762) -- (4.6567, 4.5672) -- (4.8022, 4.4945) -- (4.9477, 4.5672) -- cycle;

  %path17
  \path[draw=black,miter limit=10.0] (4.5376, 6.6278) -- (3.1253, 8.0402);

  %path18
  \path[draw=black,fill=black,miter limit=10.0] (2.971, 8.1944) -- (3.2798, 8.0915) -- (3.1253, 8.0402) -- (3.0739, 7.8856) -- cycle;

  %path31
  \path[draw=black,miter limit=10.0] (17.8229, 5.8902) -- (20.8857, 8.953);

  %ellipse31
  \path[draw=black,fill=black] (17.7668, 5.8341) circle (0.0794cm);

  %ellipse32
  \path[draw=black,fill=black] (20.9418, 9.0091) circle (0.0794cm);

  %path33
  \path[draw=black,miter limit=10.0] (14.6479, 8.953) -- (17.7107, 5.8902);

  %ellipse33
  \path[draw=black,fill=black] (14.5918, 9.0091) circle (0.0794cm);

  %ellipse34
  \path[draw=black,fill=black] (17.7668, 5.8341) circle (0.0794cm);

  %path35
  \path[draw=black,miter limit=10.0] (17.7668, 2.2093) -- (17.7668, 5.7547);

  %ellipse35
  \path[draw=black,fill=black] (17.7668, 2.1299) circle (0.0794cm);

  %ellipse36
  \path[draw=black,fill=black] (17.7668, 5.8341) circle (0.0794cm);

  %path37
  \path[draw=blue,line width=0.0529cm,miter limit=10.0] (19.0897, 0.0132) .. controls (17.6786, 1.6007) and (16.8848, 2.8795) .. (16.7084, 3.8497) .. controls (16.532, 4.8198) and (16.488, 5.6136) .. (16.5761, 6.2309) .. controls (16.6643, 6.8483) and (16.973, 7.4216) .. (17.5022, 7.9507) .. controls (18.0314, 8.4799) and (18.6928, 8.8768) .. (19.4866, 9.1414) .. controls (20.2803, 9.4059) and (21.9119, 9.3618) .. (24.3814, 9.0091);

  %path39
  \path[draw=red,line width=0.0529cm,miter limit=10.0] (17.2376, 0.2778) .. controls (18.4723, 1.3361) and (18.9574, 3.0118) .. (18.6928, 5.3049) .. controls (18.4282, 7.598) and (16.532, 9.1855) .. (13.0043, 10.0674);

  %path41
  \path[draw=c00c800,line width=0.0529cm,miter limit=10.0] (13.798, 10.8611) .. controls (13.798, 8.2153) and (14.4595, 6.4955) .. (15.7824, 5.7018) .. controls (17.1053, 4.908) and (18.4723, 4.9521) .. (19.8834, 5.8341) .. controls (21.2945, 6.716) and (22.2647, 8.2153) .. (22.7939, 10.332);

  %path43
  \path[draw=black,miter limit=10.0] (20.148, 8.2153) -- (19.5294, 7.5967);

  %path44
  \path[draw=black,fill=black,miter limit=10.0] (19.0059, 7.0732) -- (19.1089, 7.382) -- (19.1602, 7.2275) -- (19.3147, 7.1761) -- cycle;

  %path45
  \path[draw=black,miter limit=10.0] (17.7668, 2.6591) -- (17.7668, 3.4695);

  %path46
  \path[draw=black,fill=black,miter limit=10.0] (17.7668, 4.5873) -- (17.9123, 4.2963) -- (17.7668, 4.3691) -- (17.6212, 4.2963) -- cycle;

  %path47
  \path[draw=black,miter limit=10.0] (14.5918, 9.0125) -- (15.7395, 7.8616);

  %path48
  \path[draw=black,fill=black,miter limit=10.0] (15.8938, 7.707) -- (15.5853, 7.8105) -- (15.7395, 7.8616) -- (15.7914, 8.0158) -- cycle;

  \tikzstyle{every node}=[font=\fontsize{18}{18}\selectfont]
%text64
  \node[anchor=south west] (text64) at (4.921, 4.8699){$gb$};

  %text65
  \node[anchor=south west] (text65) at (2.0049, 7.9068){$rg$};

  %text66
  \node[anchor=south west] (text66) at (6.4462, 7.3526){$br$};

  %path69
  \path[draw=blue,line cap=butt,line join=miter,line width=0.0529cm,miter limit=10.0] (1.4449, 4.5372).. controls (1.4449, 4.5372) and (0.8577, 4.9448) .. (2.7383, 7.2044).. controls (4.6189, 9.464) and (4.6426, 10.7905) .. (4.6426, 10.7905);

  %path70
  \path[draw=blue,line cap=butt,line join=miter,line width=0.0529cm,miter limit=10.0] (0.2842, 7.1427).. controls (0.2842, 7.1427) and (3.4346, 9.5824) .. (2.7003, 10.4115);

  %path71
  \path[draw=c00c800,line cap=butt,line join=miter,line width=0.0529cm,miter limit=10.0] (9.9958, 7.9717).. controls (7.9047, 7.9784) and (6.1349, 7.6638) .. (7.2245, 10.4825);

  %path72
  \path[draw=c00c800,line cap=butt,line join=miter,line width=0.0529cm,miter limit=10.0] (5.7559, 10.5299).. controls (5.7559, 10.5299) and (5.8269, 7.9244) .. (10.1853, 5.6267);

  %path73
  \path[draw=c00c800,line cap=butt,line join=miter,line width=0.0529cm,miter limit=10.0] (4.9979, 9.914).. controls (4.9979, 9.914) and (5.8033, 5.826) .. (7.6034, 4.3762);

  %path74
  \path[draw=c00c800,line width=0.0529cm,miter limit=10.0] (4.273, 0.807) .. controls (5.5077, 1.8653) and (5.9928, 3.541) .. (5.7282, 5.8341) .. controls (5.4636, 8.1271) and (3.5675, 9.7146) .. (0.0397, 10.5966);

  %path75
  \path[draw=red,line cap=butt,line join=miter,line width=0.0529cm,miter limit=10.0] (18.0555, 9.7216).. controls (18.4955, 6.198) and (19.6843, 5.8841) .. (20.4453, 4.6671);

  %path76
  \path[draw=red,line cap=butt,line join=miter,line width=0.0529cm,miter limit=10.0] (19.3284, 9.9226).. controls (22.1758, 5.7018) and (20.9028, 7.4102) .. (22.3433, 4.8979);

  %path77
  \path[draw=c00c800,line cap=butt,line join=miter,line width=0.0529cm,miter limit=10.0] (14.9067, 3.7924).. controls (14.9067, 3.7924) and (17.486, 5.2998) .. (21.4053, 4.8979);

  %path78
  \path[draw=c00c800,line cap=butt,line join=miter,line width=0.0529cm,miter limit=10.0] (15.6436, 2.486).. controls (15.6436, 2.486) and (20.7353, 3.0555) .. (22.3098, 2.888);

  %path79
  \path[draw=c00c800,line cap=butt,line join=miter,line width=0.0529cm,miter limit=10.0] (13.7342, 2.2515).. controls (17.7668, 3.6878) and (15.2751, 4.3284) .. (21.8073, 4.2949);

  %text79
  \node[anchor=south west] (text79) at (17.925, 3.3615){$gb$};

  %text80
  \node[anchor=south west] (text80) at (15.6969, 6.7159){$rg$};

  %text81
  \node[anchor=south west] (text81) at (19.7677, 7.1671){$br$};

\end{tikzpicture}}
    \caption{Edge labeling and orientation scheme for vertices of degree $3$. Degree 2 and degree 1 vertices are possible as well.}
    \label{fig:orient}
\end{figure}
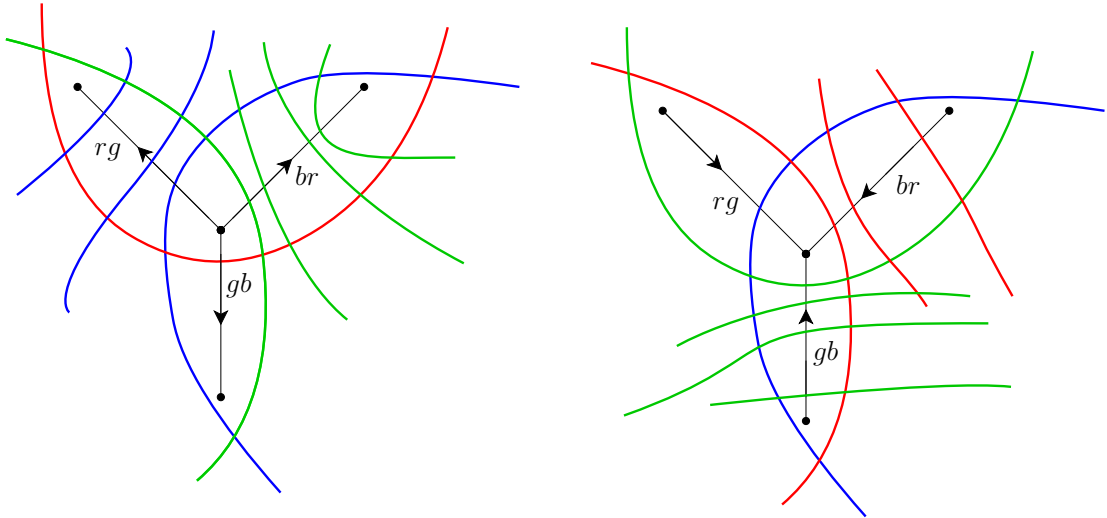

Observe that every vertex of $G$ is either a sink of incoming edges or a source of outgoing edges, depending on whether the cyclic order on the colors of the strands around the region is $(g,r,b)$ or $(b,r,g)$. 
Thus $G$ is bipartite.

Let $k'$ be the nonnegative integer such that there are exactly $k-k'$ regional bigons in $\gamma$. We finally argue that $G$ has a vertex cover of size at most $k'$ if and only if $\gamma$ has a vertex cover of size at most $k$.

Assume that $G$ has a vertex cover of size at most $k'$. Pin $\gamma$ with at most $k$ pins as follows: place $k-k'$ pins in the $k-k'$ regional bigons (whose presence is forced by the pinning property), and the remaining $k'$ pins in the triangular regions of $\gamma$ corresponding to vertices of $G$ belonging to the cover. 
Assume that $\gamma$ can be pinned with $k$ pins. Construct a vertex cover of $G$ of size at most $k'$ as follows: for each of the $k'$ innermost bigons $B$ of $\gamma$ which are not regional, there is a pin in at least one of the $n+2$ regions among $\{T_x, R_1, R_2, \ldots, R_n, T_y\}$ (note that the rectangular regions $R_i$ need not be considered when pinning $\gamma$ optimally, since one may always move their pins to one of the triangular regions $T_x$ or $T_y$ while preserving a pinning set of $\gamma$).
\end{proof}

\subsection{Pinning $4$-strand simple multiloops}
\label{sec:4-strands}

It is natural to ask whether these ideas can be extended to simple multiloops with $4$ strands. We restate Conjecture \ref{conj:4-hard}.

\begin{conjecture}
For every surface $\Sigma$, the problem $\mathsf{SimplePin}(\Sigma,4)$ is \textsf{NP}-complete.
\end{conjecture}

One may hope to prove this conjecture by reduction from \textsf{3C3PVC} as hinted in Theorem \ref{thm:simple-no-strand-restiction-hard} and Remark \ref{rmk:4color}. Given a bridgeless cubic planar graph $G$ and a valid $4$-coloring of the dual graph of $G$, one may obtain a simple multiloop $\gamma$ whose strands are $4$-colored according to the faces of $G$ they surround (in quadratic time). One could then imagine performing surgeries to connect strands of $\gamma$ of the same color, without damaging the mobidisc structure of $\gamma$ too much.

For some additional rationale for this conjecture, note that the combinatorial complexity of mobidiscs for multiloops with $4$ strands is evidently more complex than the $3$-strand case; in particular there is no clear reduction to \textsf{Vertex Cover}. The clause structure of a mobidisc \textsf{CNF} of a $3$-strand simple multiloop is predictable and controlled: after ignoring rectangular regions, the disjunctions may be thought of as having at most $2$ variables each, and the fact that the associated graph is bipartite is equivalent to all cycles having even length. In an earlier attempt at proving Theorem \ref{thm:3-easy}, we also noticed that the graph $G$ produced by the reduction has $\dim(H_1(G;\Z))\le 5$, and we believe that \textsf{Vertex Cover} for graphs of bounded homological rank is in $\textsf{P}$.

Thus, while we believe that the conjecture is true, one may hope to disprove it by bounding the homological rank or finding certain parity structures on appropriate ``hypergraph complexes'' associated to $4$-strand simple multiloops.

\section{Simple multiloops with many strands}
\label{sec:20hard}

In this section, we demonstrate the hardness of pinning simple multiloops with a fixed sufficiently large number of strands, namely we show that for $s\ge 20$ the $\mathsf{SimplePin}(\Sigma,s)$ problem is \textsf{NP}-hard.
Thus (under the assumption that \textsf{P}$\ne$\textsf{NP}) there is no polynomial-time algorithm to compute minimal pinning sets for simple multiloops with many strands.

We prove \textsf{NP}-hardness by reduction from an appropriate variant of the \textsf{Vertex Cover} problem, namely \textsf{3-Connected Cubic Planar Vertex Cover} (\textsf{3C3PVC} in Definition \ref{def:vc-prob}).

\subsection{The \textsf{3-Connected Cubic Planar Vertex Cover} problem}

We now establish that \textsf{3C3PVC} is \textsf{NP}-hard. This was essentially shown in work of Uehara \cite{uehara_cubic_NP_complete_96}, but we add some details for completeness.

\begin{lemma}[hardness of \textsf{3C3PVC}] 
The problem \textsf{3C3PVC} is \textsf{NP}-complete.
\label{lem:pvc3c_hard}
\end{lemma}

\begin{proof}
    Note that the problem is in \textsf{NP} as a special case of \textsf{Vertex Cover}.
    We show that is \textsf{NP}-hard by reduction from \textsf{Max-degree 3 Planar Vertex Cover} (the restriction of \textsf{Vertex Cover} to planar graphs of max-degree $3$), which problem is shown to be \textsf{NP}-hard in \cite[Lemma 1]{Garey-Johnson_rectilinear-Steiner-tree_1977}.
    Let $(G,k)$ be an instance of \textsf{Max-degree 3 Planar Vertex Cover} (so $G$ is a planar graph of max-degree $3$ and $k\in\Z$).

    \paragraph{Passing to a planar graph of max-degree $3$ and min-degree $2$.} We first construct a graph $G'$ from $G$ which is planar of max-degree $3$ and min-degree $2$. We do this by iteratively deleting a subgraph of $G$ induced by an edge $(v,w)$ where $v$ has degree $1$. The reader may check that each time this procedure is performed, the size of the vertex cover must be decreased by $1$. In other words, if this procedure is performed $n$ times before it terminates in a planar graph $G'$ with max-degree $3$ and min-degree $2$, then $G$ has a vertex cover of size at most $k$ if and only if $G'$ has a vertex cover of size at most $k'=k-n$.   
    
    \paragraph{Passing to a connected planar graph of max-degree $3$ and min-degree $2$.} Suppose now that $G'$ has $n$ components. We next construct a connected graph $G''$ which has a vertex cover of size at most $k''=k'+3n$ if and only if $G'$ has a vertex cover of size at most $k'$.

    We may assume that every connected component of $G'$ has at least one edge (since an isolated vertex of $G'$ never belongs to a minimal vertex cover, so discarding them yields a graph with the same minimal vertex covers as $G'$), and we may also assume $n>0$.

    We construct $G''$ from $G'$ as follows (see Figure \ref{fig:may_assume_connected}).
    For each connected component $G_i$ of $G'$, choose an edge $e_i$ and let $v_i,w_i$ be its endpoints. Subdivide $e_i$ by adding two vertices $x_i,y_i$, and add another vertex $z_i$ connected to $x_i$ and $y_i$ to form a triangle. Let $C_{2n}$ be a cycle of length $2n$ with vertices labeled $b_1,c_1,\ldots b_n,c_n$ in this cyclic order.
    Add edges connecting each $z_i$ to $b_i$. Finally, subdivide each edge connecting $z_i$ to $b_i$ by adding a vertex $a_i$.
    
    This achieves the construction of $G''$. The reader may check that it is planar and connected with max-degree $3$ and min-degree $2$, and it contains all vertices of $G'$.

    \begin{figure}[H]
    \centering
    \scalebox{0.8}{\input{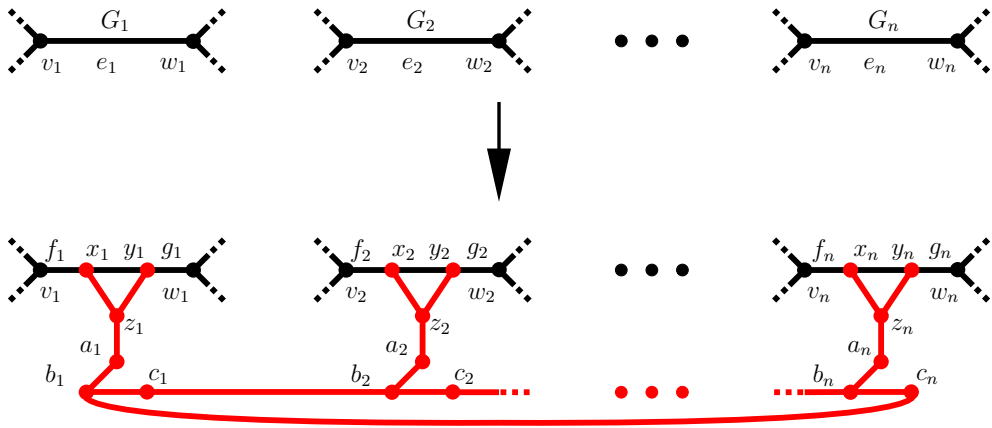}}
    \caption{From disconnected $G'$ (top) to connected $G''$ (bottom).}
    \label{fig:may_assume_connected}
    \end{figure}

    Now we show that $G''$ has a vertex cover of size at most $k''=k'+3n$ if and only if $G'$ has vertex cover of size at most $k'$.
    
    Suppose that $G'$ has a vertex cover of size at most $k'$ and let us build a cover of size at most $k''$ of $G''$ containing all vertices from the initial cover of $G'$. For each $i$, at least one of $v_i,w_i$ belong to the cover of $G'$: if $v_i$ does then we add $y_i,z_i,b_i$ to the cover of $G''$ but otherwise $w_i$ does so we add $x_i,z_i,b_i$. 
    This yields a vertex cover of $G''$ by $k''$ vertices.
    
    Suppose $G''$ has a vertex cover of size at most $k''$. For each $i$, at least $3$ vertices from $\{x_i,y_i,z_i,a_i,b_i,c_i\}$ must belong to the cover.
    We argue that any cover of $G''$ can be modified to obtain a cover with the same cardinal such that it contains for each $i$ at least one of the vertices $v_i$ or $w_i$.

    For each $i$, we modify the cover according to the following dichotomy.
    If neither $v_i$ nor $w_i$ belong to the cover, then both $x_i$ and $y_i$ belong to the cover, and the cover must contain at least $2$ other vertices from the set $\{z_i,a_i,b_i,c_i\}$: we thus replace these $4$ vertices by $\{v_i,y_i,z_i,b_i\}$, hence covering all edges in the subgraph induced by $\{v_i,w_i,x_i,y_i,z_i,a_i,b_i,c_i\}$. 
    If at least one of $v_i$ or $w_i$ belongs to the cover, then we replace any vertices from the subset $\{z_i,a_i,b_i,c_i\}$ occurring in the cover (at least two of which must belong to the cover) with the vertices $z_i$ and $b_i$. 

    This completes the modification of the cover.
    By construction, this new vertex set covers the edges of the graph induced by $\{v_i,w_i,x_i,y_i,z_i,a_i,b_i,c_i\}$.
    Note that since each $b_i$ belongs to the cover, the even cycle $C_{2n}$ remains covered after this modification.
    Denote by $H$ the subgraph of $G''$ induced by $\bigcup_{i=1}^n\{x_i,y_i,z_i,a_i,b_i,c_i\}$ (in red on Figure \ref{fig:may_assume_connected}), and denote $f_i=(v_i,x_i)$ and $g_i=(w_i,y_i)$ the edges of $G''$.
    From the modified cover, the subgraph $H$ contains at least $3n$ vertices and the subgraph $G''\setminus (H\cup\bigcup_{i=1}^n\{f_i,g_i\})$ is covered by a set of at most $k'$ vertices which for every $i$ contains $v_i$ or $w_i$. This shows that $G'$ admits a vertex cover of size at most $k'$.

    \paragraph{Passing to a connected cubic planar graph.} We now construct a graph $G'''$ from $G''$ which is cubic. Every vertex of $G'$ has degree $2$ or $3$. We replace each vertex of degree $2$ with a square and diagonal (see Figure \ref{fig:make_cubic}). Let $n$ be the number of times that we made this replacement to produce $G'''$. The reader may check that $G'''$ is cubic, connected, and has a vertex cover of size at most $k'''=k''+2n$ if and only if $G''$ has a vertex cover of size at most $k''$.

    \begin{figure}[H]
    \centering
    \scalebox{1.5}{\def \globalscale {1.000000}
\begin{tikzpicture}[y=1cm, x=1cm, yscale=\globalscale,xscale=\globalscale, every node/.append style={scale=\globalscale}, inner sep=0pt, outer sep=0pt]
  \begin{scope}[shift={(-4.9929, 9.144)}]%% layer1
    \begin{scope}[cm={ -1.0,-0.0,0.0,-1.0,(10.16, -16.256)}]%% g21
      %path1
      \path[draw=black,line cap=butt,line join=miter,line width=0.0765cm] (5.08, -7.62) -- (5.08, -8.128) -- (5.08, -8.636);

      %path8
      \path[draw=black,fill=black,even odd rule,line cap=round,line width=0.0cm] (5.08, -8.128) circle (0.0871cm);

      %path16
      \path[draw=black,line cap=butt,line join=miter,line width=0.0765cm,dash pattern=on 0.0383cm off 0.0383cm] (5.08, -7.366) -- (5.08, -7.62);

      %path17
      \path[draw=black,line cap=butt,line join=miter,line width=0.0765cm,dash pattern=on 0.0383cm off 0.0383cm] (5.08, -8.89) -- (5.08, -8.636);

    \end{scope}
    %path2
    \path[draw=black,line cap=butt,line join=miter,line width=0.0765cm,miter limit=4.0] (7.366, -8.89) -- (7.366, -8.636);

    %path3
    \path[draw=red,line cap=butt,line join=miter,line width=0.0765cm,miter limit=4.0] (7.366, -8.636) -- (6.604, -8.128);

    %circle9
    \path[draw=red,fill=red,even odd rule,line cap=round,line width=0.0cm,rotate around={-180.0:(0.0, 2.032)}] (-6.604, 12.192) circle (0.0871cm);

    %path18
    \path[draw=black,line cap=butt,line join=miter,line width=0.0765cm,dash pattern=on 0.0383cm off 0.0383cm] (7.366, -9.144) -- (7.366, -8.89);

    %path19
    \path[draw=black,line cap=butt,line join=miter,line width=0.0765cm,dash pattern=on 0.0383cm off 0.0383cm] (7.366, -7.112) -- (7.366, -7.366);

    \begin{scope}[shift={(0.5972, -0.3762)}]%% g1
      %path121-0
      \path[draw=black,line cap=butt,line join=miter,line width=0.048cm,miter limit=10.0,cm={ 0.2646,-0.0,-0.0,0.2646,(-4.6223, -4.073)}] (35.968, -13.904) -- (38.144, -13.904);

      %path123
      \path[fill=black,even odd rule,cm={ 0.2646,-0.0,-0.0,0.2646,(-4.6223, -4.073)}] (37.152, -13.68) -- (38.144, -13.904) -- (37.152, -14.128) -- cycle;

      %path125
      \path[draw=black,line cap=butt,line join=miter,line width=0.048cm,miter limit=10.0,cm={ 0.2646,-0.0,-0.0,0.2646,(-4.6223, -4.073)}] (37.152, -13.68) -- (38.144, -13.904) -- (37.152, -14.128);

    \end{scope}
    %path12
    \path[draw=black,line cap=butt,line join=miter,line width=0.0765cm,miter limit=4.0] (7.366, -7.62) -- (7.366, -7.366);

    %circle16
    \path[draw=red,fill=red,even odd rule,line cap=round,line width=0.0cm,rotate around={-180.0:(0.0, 2.032)}] (-8.128, 12.192) circle (0.0871cm);

    %path22
    \path[draw=red,line cap=butt,line join=miter,line width=0.0765cm,miter limit=4.0] (8.128, -8.128) -- (7.366, -8.636);

    %path23
    \path[draw=red,line cap=butt,line join=miter,line width=0.0765cm,miter limit=4.0] (7.366, -7.62) -- (6.604, -8.128);

    %path24
    \path[draw=red,line cap=butt,line join=miter,line width=0.0765cm,miter limit=4.0] (8.128, -8.128) -- (7.366, -7.62);

    %circle8
    \path[draw=red,fill=red,even odd rule,line cap=round,line join=miter,line width=0.0cm,miter limit=4.0,rotate around={-180.0:(0.0, 2.032)}] (-7.366, 12.7) circle (0.0871cm);

    %circle1
    \path[draw=red,fill=red,even odd rule,line cap=round,line join=miter,line width=0.0cm,miter limit=4.0,rotate around={-180.0:(0.0, 2.032)}] (-7.366, 11.684) circle (0.0871cm);

    %path25
    \path[draw=red,line cap=butt,line join=miter,line width=0.0765cm,miter limit=4.0] (8.128, -8.128) -- (6.604, -8.128);

  \end{scope}

\end{tikzpicture}}
    \caption{From $G''$ to cubic $G'''$.}
    \label{fig:make_cubic}
    \end{figure}
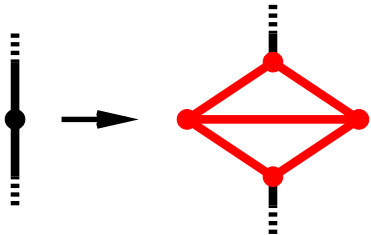

    \paragraph{Passing to a $3$-connected cubic planar graph.} We finally apply Uehara's gadget which is shown in Figure \ref{fig:uehara_gadget} as many times as needed until we have a $3$-connected graph $G''''$, starting with edges of $G'''$ which are bridges (edges whose removal disconnects the graph). This creates an intermediate graph that is $2$-connected. Then we repeatedly apply it once along one edge from a pair of edges whose removal disconnects the graph (such pairs can be detected in polynomial time using Menger's theorem). The reader may check that this process terminates in polynomial time.
    
    \begin{figure}[H]
    \centering
    \scalebox{1.5}{\input{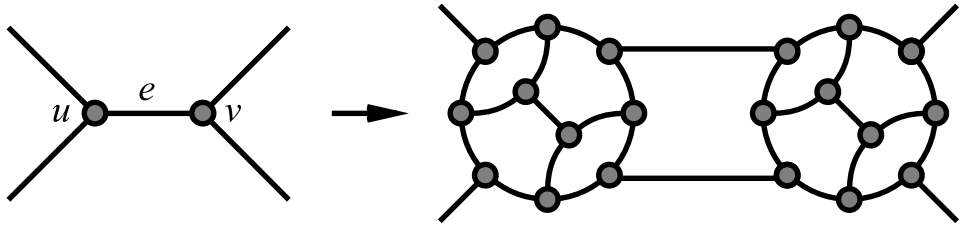}}
    \caption{From $G'''$ to $G''''$ via the Uehara connectivity-increasing gadget (reproduced from \cite{uehara_cubic_NP_complete_96} with permission).}
    \label{fig:uehara_gadget}
    \end{figure}
    
    Uehara \cite{uehara_cubic_NP_complete_96} shows that if this gadget was added $n$ times to form $G''''$, then $G''''$ has a vertex cover of size at most $k''''=k'''+10n$ if and only if $G'''$ has a vertex cover of size at most $k'''$. \end{proof}

\subsection{Planar graph embeddings with few slopes}

To prove Theorem \ref{thm:20-hard}, we will reduce from \textsf{3C3PVC}. 
The first step of the reduction is critically important to our construction and relies on the following lemma.

\begin{lemma}[embedding $3$-connected cubic planar graphs with exactly $6$ slopes]
\label{lem:embed_with_4_slopes}
    Every $3$-connected cubic planar graph $G$ has a straight-line planar embedding, computable in linear time, in which the union of edge slopes is a $6$-element set $\{\frac{\pi}{4}, \frac{\pi}{2}, \frac{3\pi}{4},\phi_1,\phi_2,\phi_3\}$ for some $\phi_1,\phi_2,\phi_3\in [0,\pi)$, and each of $\phi_1, \phi_2, \phi_3$ is realized exactly once by a single edge on the outer face of $G$. Additionally, the $3$ edges with slopes $\phi_1, \phi_2, \phi_3$ occur in succession on the outside face. Moreover, one can arrange that the embedding is isotopic to any given planar embedding of $G$. In particular one can specify the outer face in advance.
\end{lemma}

\begin{proof}
This lemma is proved in \cite{dumovic_et_al_few_slopes_07} without an explicit complexity analysis. The first step is to compute a so-called canonical decomposition of $G$ based on an initial planar embedding, which may be done in linear time \cite{Kant1996}[Theorem 3.2] (an initial planar embedding may be computed in linear time by \cite{Hopcroft-Tarjan_planar-embedding_1974}). A careful reading of the recursive construction in the proof of \\ \cite{dumovic_et_al_few_slopes_07}[Theorem 24] reveals that a drawing of $G$ may be computed in linear time as well. In the very last step, it is sometimes possible that we may be able finish the drawing so that fewer than $6$ slopes are used, but vertically perturbing the final vertex allows us to realize exactly $6$ slopes. This last step also shows that the $3$ edges with slopes $\phi_1, \phi_2, \phi_3$ occur in succession on the outside face.
\end{proof}

\subsection{Pinning simple multiloops with \texorpdfstring{$20$}{20} strands is \textsf{NP}-complete}
\label{subsec:20hard}

Given an instance $(G,k)$ of \textsf{3C3PVC} we will construct an instance $(\gamma,k')$ of $\mathsf{SimplePin}(\Sigma,20)$ in polynomial time, and show that they are equivalent problems.

\subsubsection{Construction summary}

Roughly speaking, we will embed $G$ in a small open set homeomorphic to the plane and then meticulously assemble a simple multiloop $\gamma$ with $20$ strands around the graph in such a way that each edge of the graph is contained in a thin bigon, so that the innermost-mobidisc formula for $\gamma$ is predictable from the structure of $G$.

Here is a summary of the entire construction. We work in the Euclidean plane $\R^2$ with distance function $d$. For a subset $X\subset \R^2$ and $\epsilon\in\,\R_+$ we denote by $\mathbb{B}(X,\epsilon)$ the open $\epsilon$-neighborhood of $X$.

\begin{enumerate}
    \item  Use Lemma \ref{lem:embed_with_4_slopes} to construct a straight-line embedding of $G$ in $\R^2$ in polynomial time in which the union of edge slopes is a $6$-element set $\{\frac{\pi}{4}, \frac{\pi}{2}, \frac{3\pi}{4},\phi_1,\phi_2,\phi_3\}$ for some $\phi_1,\phi_2,\phi_3\in [0,\pi)$, and each of $\phi_1$, $\phi_2$ and $\phi_3$ is realized exactly once by a single edge on the outer face of $G$. See Figure \ref{fig:construction_summary} left and middle. A schematic of the entire construction is shown in Figure \ref{fig:construction_summary} at right. Scale the embedding so that $\mathbb{B}(0,\frac{1}{2})$ contains $G$ as well as all intersections between all lines defined by distinct edges of $G$. 

    \begin{figure}[H]
    \centering
    \scalebox{0.44}{\input{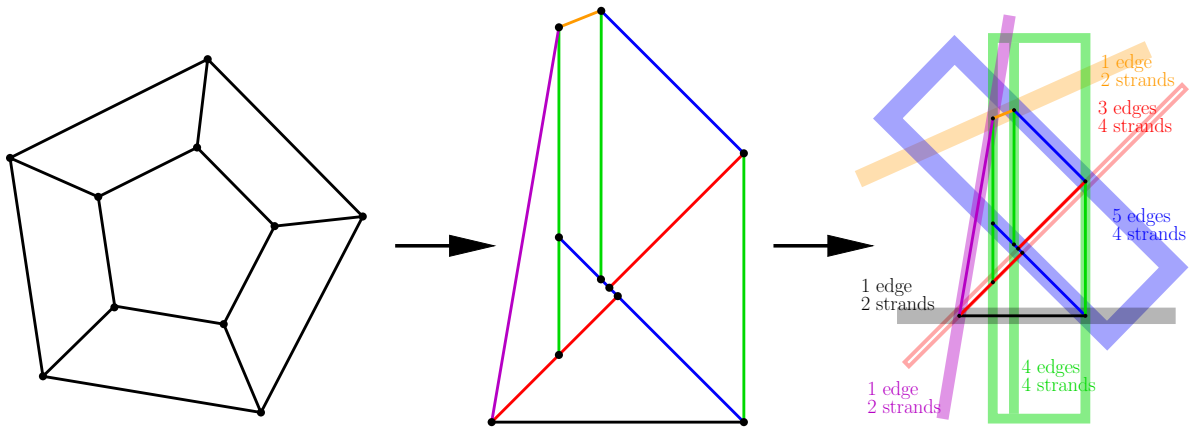}}
    \caption{Left: Input 3-connected cubic planar graph $G$. Middle: A drawing of $G$ with $6$ slopes from the algorithm in \cite{dumovic_et_al_few_slopes_07}. Right: A ``plumbing'' scheme to construct the $20$-strand simple multiloop $\gamma$ associated to $G$ (without the rescaling step).
    For each slope (color), we see a neighborhood of the (at most $4$) strands associated to that slope, which will contain a ``plumbing'' gadget as in Figures \ref{fig:alphaonly}, \ref{fig:bothstrands}, \ref{fig:betaonly}, \ref{fig:nostrands}. We do not show that the two anchor strands will surround the multiloop as in Figure \ref{fig:edge_gadget_bundle}.}
    \label{fig:construction_summary}
    \end{figure}
    
    \item Each edge will be encased in a thin bigon called the edge gadget (see Figure \ref{fig:edgegadget}) which will use two strands of $\gamma$. Each set of edges which are collinear will give rise to a sequence of edge gadgets alternating between the two types ($\alpha$ and $\beta$) to form the edge gadget bundle which bisects the unit circle along the line defined by the edges (see Figure \ref{fig:edge_gadget_bundle}).

    \item The edges of $G$ of a given slope $m\in\{\frac{\pi}{4}, \frac{\pi}{2}, \frac{3\pi}{4},\phi_1,\phi_2,\phi_3\}$ are partitioned into sets of collinearity classes, each with an edge gadget bundle consisting of $2$ or $4$ strands of $\gamma$. They will be plumbed together into a slope $m$ plumbing sub-multiloop $\gamma_m$ which uses exactly $6$ strands, $2$ of which are anchor strands near the unit circle. A parity condition leads to distinguishing cases in the construction: examples of all cases are shown in Figures \ref{fig:alphaonly}, \ref{fig:bothstrands}, \ref{fig:betaonly}, \ref{fig:nostrands}. Each of the $6$ slopes yields $4$ strands, and with the $2$ anchor strands we have a total $26$ strands.

    \item We tweak the parameters of our construction to ensure that all intersections of strands are transverse and there is no unwanted interaction between gadgets, and resolve intersection points of higher multiplicity locally to obtain a multiloop with $26$ strands. We then remove $6$ unnecessary strands corresponding to the slopes $\phi_1, \phi_2, \phi_3$ to obtain a multiloop $\gamma$ with $20$ strands.

    \item We define the number $k'$ and show that the size of $(\gamma,k')$ in memory is polynomial in that of $(G,k)$. Finally we study the mobidisc structure of $\gamma$ obtained from this construction and show that it is pinned with $k'$ pins if and only if $G$ has a vertex cover of size at most $k$.

\end{enumerate}

\subsubsection{Precise descriptions of edge gadgets and edge gadget bundles}

First, we define the \emph{edge gadget} represented in Figure \ref{fig:edgegadget}.
\begin{definition}[edge gadget]
\label{edge-gadget}
The \emph{edge gadget} (with parameter a real $\epsilon>0$, naturals $\eta,\eta' \in \N$, and a letter from the set $\{\alpha,\beta\}$) is associated to an oriented straight edge $e$ with slope $m$ from point $x$ to point $y$ in $\R^2$. It is the planar configuration in Figure \ref{fig:edgegadget} of two piecewise linear paths with $3$ corners each (with labels $\{\alpha_m^+,\alpha_m^-\}$ colored orange and purple respectively, or labels $\{\beta_m^+,\beta_m^-$\} colored green and blue respectively) forming a bigon $M_e$ around the oriented edge. Let $\delta=\epsilon\tan\epsilon$ and note that as $\epsilon\to 0$, the angles $\theta_1,\theta_2,\theta_3,\theta_4\to 0$ and the bigon $M_e$ approaches $e$.
\end{definition}

\begin{figure}[H]
    \centering
     \scalebox{0.4}{\input{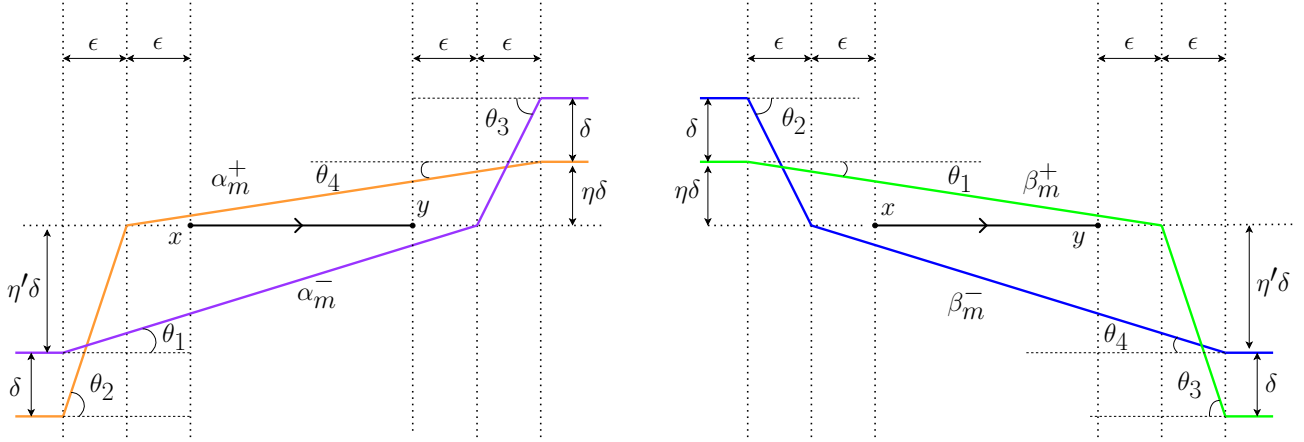}}
    \caption{Left: The edge gadget with parameter $\alpha$. Right: The edge gadget with parameter $\beta$.}
    \label{fig:edgegadget}
\end{figure}

Next, we define the \emph{edge gadget bundle} represented in Figure \ref{fig:edge_gadget_bundle}.
\begin{definition}[edge gadget bundle] \label{def:edge-gadget-bundle}
The edge gadget bundle with parameter $\epsilon>0$ is a planar configuration associated to a set of $n\geq 1$ collinear edges in $\mathbb{B}(0,\frac{1}{2})$ having slope $m$. It consists of $2n$ piecewise linear paths with labels among $\alpha_m^+,\alpha_m^-,\beta_m^+, \beta_m^-$ (respectively colored orange, purple, green, blue), the unit circle $\mu$, and another strand $\nu$ which is within a closed $\delta=\epsilon\tan\epsilon$-neighborhood of $\mu$. To describe it, we rotate the plane so that the slope $m$ is vertical, orient the edges to point upward, enumerate them $\{e_i\}_{i=1}^n$ from bottom to top, and let $L$ be the oriented line containing them. Construct edge gadgets with parameter $\epsilon$ around each edge, alternating between types $\alpha$ and $\beta$, starting with an $\alpha$ edge gadget at the bottom (see Figure \ref{fig:edge_gadget_bundle} at left). For each edge gadget, choose distinct $\eta$ and $\eta'$ so that
\begin{enumerate}[noitemsep]
    \item all vertical segments of both $\alpha$ and $\beta$ edges are pairwise non-collinear and occur at horizontally-spaced increments of $\delta=\epsilon\tan\epsilon$, 
    \item paths of $\alpha$ edges join monotonously the top-left to the bottom-right,     
    paths of $\beta$ edges join monotonously the top-right to the bottom-left, and
    \item all vertical $\beta$ segments are closer than the $\alpha$ segments to $L$.
\end{enumerate}

Extend all upward-pointing vertical paths of edge gadgets upward to distance $\frac{1}{2}$ beyond $\mu$. These endpoints will be referred to as \emph{terminals}. Let $\nu$ consist of a union of arcs of $\partial\mathbb{B }(0,1-\frac{\delta}{3})$ except where \emph{teeth} bigons occur: they are peninsular subarcs shown in Figure \ref{fig:edge_gadget_bundle} at right consisting of two short vertical arcs at distance $\frac{\delta}{3}$ and $\frac{2\delta}{3}$ to the left of where each vertical arc associated to an edge gadget leaves $\mu$ which are connected by a short circular arc of $\partial\mathbb{B}(0,1+\frac{\delta}{3})$. Repeat this construction for downward-pointing vertical paths of edge gadgets after rotating the picture by $\pi$ radians.

Note that each edge $e_i$ gives rise to exactly $8$ bigons bounded by $\mu$ and $\nu$: $4$ \emph{regional tooth} bigons and $4$ \emph{subdivided teeth}: embedded bigons bissected by a path with label among $\{\alpha_m^+,\alpha_m^-,\beta_m^+,\beta_m^-\}$ with $2$ regions apiece. Thus, in addition to the $M_{e_i}$ bigons, each edge gadget bundle gives rise to an additional $8n$ bigons which will require $8n$ pins to make $\gamma$ taut. Note that the union of all of these bigons approaches $L$ as $\epsilon\to 0$ and that there will be no paths with $\beta$ labels if $n=1$. 

\begin{figure}[H]
    \centering
    \scalebox{0.7}{\input{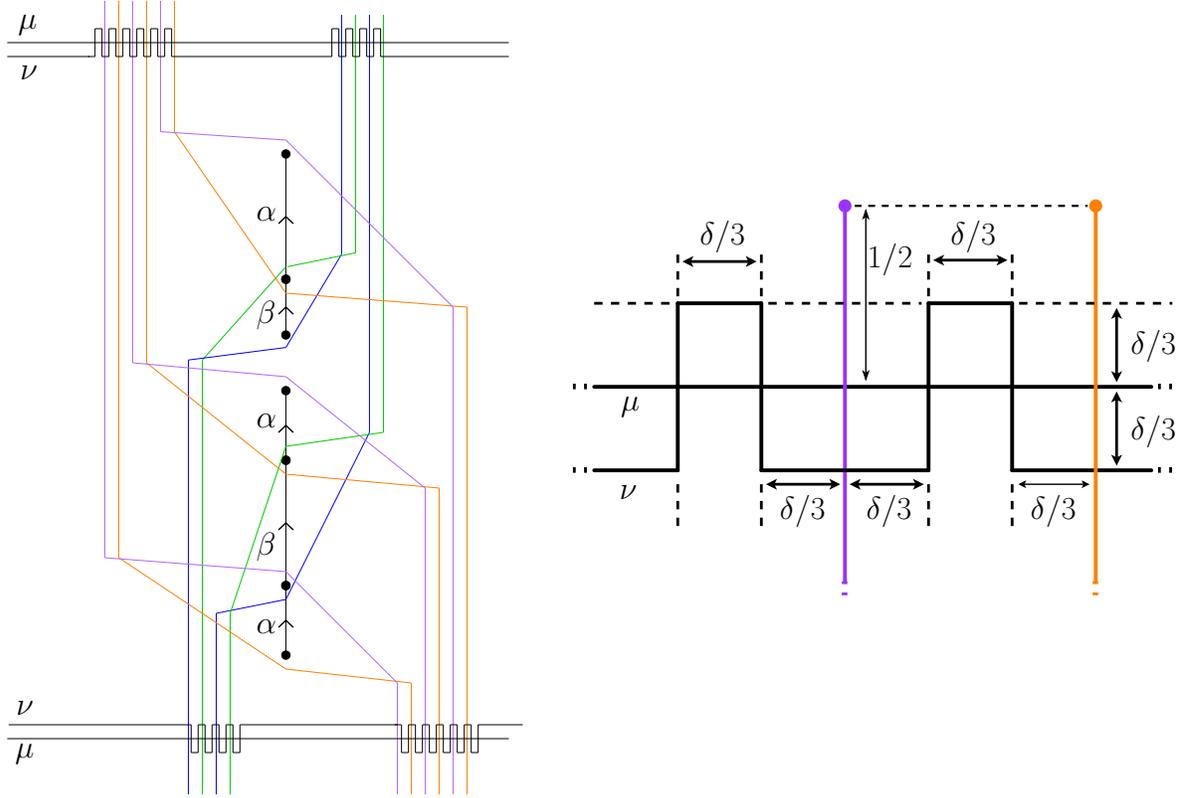}}
    \caption{Left: An edge gadget bundle with $n=5$ edges. Right: Zoomed-in view of regional teeth and subdivided teeth with terminals emphasized. Circular arcs are shown as horizontal segments.}
    \label{fig:edge_gadget_bundle}
\end{figure}

\end{definition}

\subsubsection{Connecting paths of edge gadget bundles with a given slope to form $4$ strands}
\label{sec:plumbing}

In this subsection, we explain how to connect the paths with labels $\alpha_m^+,\alpha_m^-,\beta_m^+, \beta_m^-$ from all edge gadget bundles with a given slope $m\in[0,\pi)$ to form a simple multiloop with $4$ strands.
The connections will be made outside of the unit circle, around which there will be $2$ additional strands (which will be the same for all slopes). 
The description of the construction will span several pages. Schematic examples of the end results are shown in Figures \ref{fig:alphaonly}, \ref{fig:bothstrands}, \ref{fig:betaonly}, \ref{fig:nostrands}.
The resulting construction (a multiloop $\gamma_m$ with $6$ strands) will be referred to as the \emph{bundle plumbing} or \emph{slope $m$ plumbing sub-multiloop}.

First, we orient all edges in the same direction and rotate the plane so that $m$ is vertical and the edges point upward.
The $n$ edges are partitioned into a set of collinearity classes which define parallel lines. For each collinearity class, construct an edge gadget bundle as in Definition \ref{def:edge-gadget-bundle} with parameter $\epsilon$. Assume that $\epsilon$ is small enough so that the non-vertical paths with labels among $\{\alpha_m^+,\alpha_m^-,\beta_m^+, \beta_m^-\}$ (respectively colored orange, purple, green, and blue) associated to each edge gadget bundle are contiguous. Suppose that there are $n_\alpha$ edges with $\alpha$ labels and $n_\beta$ edges with $\beta$ labels so that $n=n_\alpha+n_\beta$. Note that $n_\alpha\geq 1$ and $n_\beta\geq 0$.

\paragraph{Parity balancing procedure.}

We must ensure that the parity of each label among $\{\alpha_m^+,\alpha_m^-,\beta_m^+,\beta_m^-\}$ is even and each label occurs at least once by adding either $2$ or $4$ or $6$ more vertical segments which extend to a distance of $\frac{1}{2}$ above and below the unit circle $\mu$ at incremental distances $\delta=\epsilon\tan\epsilon$ to the left of the leftmost edge gadget bundle, and give them labels from among $\{\alpha_m^+,\alpha_m^-,\beta_m^+, \beta_m^-\}$ (still colored orange, purple, green, blue) from left to right, according to the following rules.

\begin{enumerate}
\item If $n_\alpha$ is odd and $n_\beta$ is even, add six new line segments with labels/colors \\ $(\beta_m^-,\beta_m^+,\beta_m^-,\beta_m^+,\alpha_m^-,\alpha_m^+)$/(blue, green, blue, green, purple, orange) (see Figure \ref{fig:alpha_odd_beta_even}).
\item If $n_\alpha$ is odd and $n_\beta$ is odd, add four new line segments with labels/colors \\ $(\beta_m^-,\beta_m^+,\alpha_m^-,\alpha_m^+)$/(blue, green, purple, orange) (see Figure \ref{fig:alpha_odd_beta_odd}).
\item If $n_\alpha$ is even and $n_\beta$ is odd, add two new line segments with labels/colors \\$(\beta_m^-,\beta_m^+)$/(blue, green) (see Figure \ref{fig:alpha_even_beta_odd}).
\item If $n_\alpha$ is even and $n_\beta$ is even, add four new line segments with labels/colors \\ $(\beta_m^-,\beta_m^+,\beta_m^-,\beta_m^+)$/(blue, green, blue, green) (see Figure \ref{fig:alpha_even_beta_even}).
\end{enumerate}

For each of those added segments, we also add a regional tooth between $\mu$ and $\nu$ to its left as it exits $\mu$ (thus creating $4$ bigons between $\mu$ and $\nu$) as in Figure \ref{fig:edge_gadget_bundle} at right. 

In Figures  \ref{fig:alpha_odd_beta_even}, \ref{fig:alpha_odd_beta_odd}, \ref{fig:alpha_even_beta_odd}, \ref{fig:alpha_even_beta_even}, circular arcs appear as horizontal lines, and dotted vertical lines separate edge gadget bundles from the segments added in the parity balancing procedure, but $\nu$ is not shown.

\begin{figure}[H]
    \centering
     \scalebox{1}{\input{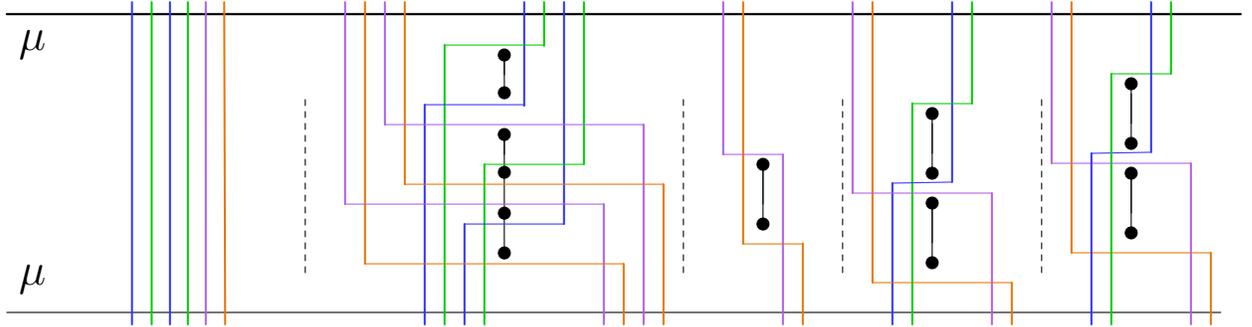}}
    \caption{Parity balancing procedure when $n_\alpha$ is odd and $n_\beta$ is even.}
    \label{fig:alpha_odd_beta_even}
\end{figure}

\begin{figure}[H]
    \centering
     \scalebox{1.05}{\input{images/alpha_odd_beta_odd}}
    \caption{Parity balancing procedure when $n_\alpha$ is odd and  $n_\beta$ is odd.}
    \label{fig:alpha_odd_beta_odd}
\end{figure}

\begin{figure}[H]
    \centering

     \scalebox{0.85}{\input{images/alpha_even_beta_odd}}
    \caption{Parity balancing procedure when $n_\alpha$ is even and $n_\beta$ is odd.}
    \label{fig:alpha_even_beta_odd}
\end{figure}

\begin{figure}[H]
    \centering
     \scalebox{0.8}{\input{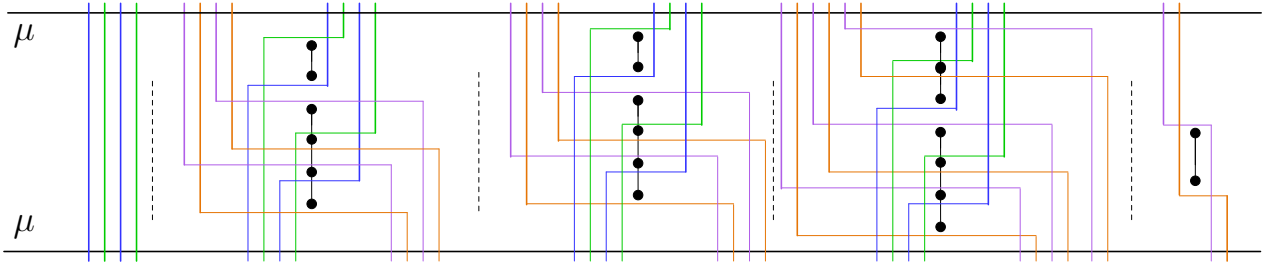}}
    \caption{Parity balancing procedure when $n_\alpha$ is even and $n_\beta$ is even.}
    \label{fig:alpha_even_beta_even}
\end{figure}

\paragraph{Upper and lower terminal pairings.}

The endpoints of paths which terminate at distance $\frac{1}{2}$ outside of $\mu$ have labels among $\alpha_m^+,\alpha_m^-,\beta_m^+, \beta_m^-$ which we call \emph{terminals}, and by reading them from left to right above and below $\mu$ we obtain two \emph{terminal sequences}: the \emph{upper terminal sequence} above $\mu$ and the \emph{lower terminal sequence} below $\mu$. Moreover, each sequence begins with $(\beta_m^-,\beta_m^+,\ldots)$ (including the case where $n_\beta= 0$).

In every terminal sequence, for each color $\sigma_m^\star \in\{\alpha_m^+,\alpha_m^-,\beta_m^+,\beta_m^-\}$, the parity balancing procedure ensures that there is a positive even number of occurrences of $\sigma_m^\star$ in the subsequence. Enumerate the positions of the subsequence $\rho_1,\ldots, \rho_{2n(\sigma_m^\star)}$ from left to right, where $2n(\sigma_m^\star)$ is the number of terminals with that label. 
Pair the terms in the subsequence as follows, for each $i\bmod{2n(\sigma_m^\star)}$:
\begin{itemize}[noitemsep]
    \item In each upper terminal sequence, pair $\rho_{2i}$ with $\rho_{2i+1}$.

    \item In each lower terminal sequence, pair $\rho_{2i-1}$ with $\rho_{2i}$.
\end{itemize}
In particular $\rho_1$ is paired with $\rho_{2n(\sigma_m^\star)}$. 

Note that connecting all terminals according to their pairings yields a closed loop and therefore a strand associated to $\sigma_m^\star\in\{\alpha_m^+,\alpha_m^-,\beta_m^+,\beta_m^-\}$. The plumbing we will construct will realize this pairing in the plane while avoiding self intersections within each strand and accidental bigons between distinct strands.

\paragraph{Blocks and pairing types.} 

We partition each terminal subsequence into \emph{blocks}: there is a block for each maximal subsequence of consecutive labels within the same edge gadget, and there is one additional block for all those coming from the parity balancing procedure. With this terminology in place, the terminal pairings may be one of the following pairing types.

\begin{itemize}[noitemsep]
    \item The pairing between $\rho_1$ and $\rho_{2n(\sigma_m^\star)}$ is called a \emph{cnidarian head} pairing.
    \item We refer to a terminal pairing as an \emph{inter-bundle} pairing when the terminals are not $\{\rho_1,\rho_{2n(\sigma_m^\star)}\}$ and come from different blocks in the terminal subsequence. 
    \item We refer to a terminal pairing as an \emph{intra-bundle} pairing when the terminals come from the same block in the terminal subsequence.
\end{itemize}

\paragraph{Connecting terminal pairings at appropriate levels.} 

We now connect the terminal points of each pairing using two vertical segments extending the vertical segments to which they belong to a circle $\partial\Ball(0,r)$ of an appropriate radius $r$, and a \emph{connecting arc} of $\partial\Ball(0,r)$ between them. We refer to the endpoints of a connecting arc as its \emph{connection points}. The radius $r$ and \emph{level} among \emph{upper levels} $\{U3,U2,U1\}$ or \emph{lower levels} $\{L1,L2\}$ of each connecting arc depend both on the type of the pairing connection and the label $\sigma_m^\star\in\{\alpha_m^+,\alpha_m^-,\beta_m^+,\beta_m^-\}$ of the pairing terminals, according to Table \ref{fig:routing_pairings}.

\begin{figure}[H]
\centering
    \begin{tabular}[h]{c|c||c|c} 
     Pairing type & Label/color & Level & Circle containing \\
     & & & associated connecting arc \\
     \hline\hline
     cnidarian head & $\beta_m^-$/blue & $U3$ &  $\partial\mathbb{B}(0,11)$ (upper arc)\\
     cnidarian head & $\beta_m^+$/green & $U3$ &  $\partial\mathbb{B}(0,10)$ (upper arc)\\
     cnidarian head & $\alpha_m^-$/purple  & $U3$ & $\partial\mathbb{B}(0,9)$ (upper arc)\\
     cnidarian head & $\alpha_m^+$/orange & $U3$ &  $\partial\mathbb{B}(0,8)$ (upper arc)\\
     \hline
     inter-bundle (upper) & $\beta_m^-$/blue & $U2$ &  $\partial\mathbb{B}(0,7)$ (upper arc)\\
     inter-bundle (upper) & $\beta_m^+$/green & $U2$ &  $\partial\mathbb{B}(0,6)$ (upper arc)\\
     inter-bundle (upper) & $\alpha_m^-$/purple  & $U2$ &  $\partial\mathbb{B}(0,5)$ (upper arc)\\
     inter-bundle (upper) & $\alpha_m^+$/orange  & $U2$ &  $\partial\mathbb{B}(0,4)$ (upper arc)\\
     \hline
     intra-bundle (upper) & $\alpha_m^-$/purple or $\beta_m^-$/blue & $U1$ &  $\partial\mathbb{B}(0,3)$ (upper arc)\\
     intra-bundle (upper) & 
     $\alpha_m^+$/orange or $\beta_m^+$/green & $U1$ &  $\partial\mathbb{B}(0,2)$ (upper arc)\\
     \hline
     \hline
     intra-bundle (lower) & $\alpha_m^+$/orange or $\beta_m^+$/green  & $L1$ &  $\partial\mathbb{B}(0,2)$ (lower arc)\\
     intra-bundle (lower) & $\alpha_m^-$/purple or $\beta_m^-$/blue & $L1$ &  $\partial\mathbb{B}(0,3)$ (lower arc)\\
     \hline     
     inter-bundle (lower) & $\alpha_m^+$/orange  & $L2$ &  $\partial\mathbb{B}(0,4)$ (lower arc)\\
     inter-bundle (lower) & $\alpha_m^-$/purple & $L2$ &  $\partial\mathbb{B}(0,5)$ (lower arc)\\
     inter-bundle (lower) & $\beta_m^+$/green & $L2$ &  $\partial\mathbb{B}(0,6)$ (lower arc)\\
     inter-bundle (lower) & $\beta_m^-$/blue & $L2$ &  $\partial\mathbb{B}(0,7)$ (lower arc)\\
     
\end{tabular}
\caption{Routing table from pairing type and label to level and connecting arc.}
    \label{fig:routing_pairings}
\end{figure}

Figures \ref{fig:alphaonly}, \ref{fig:bothstrands}, \ref{fig:betaonly}, \ref{fig:nostrands} show examples of how this procedure plays out for all of the cases in the parity balancing procedure. In these figures, connecting arcs are shown as horizontal lines, and $\nu$ is not shown. This completes the construction of the slope $m$ plumbing sub-multiloop $\gamma_m$.

\begin{figure}[H]
    \centering
     \scalebox{0.9}{\input{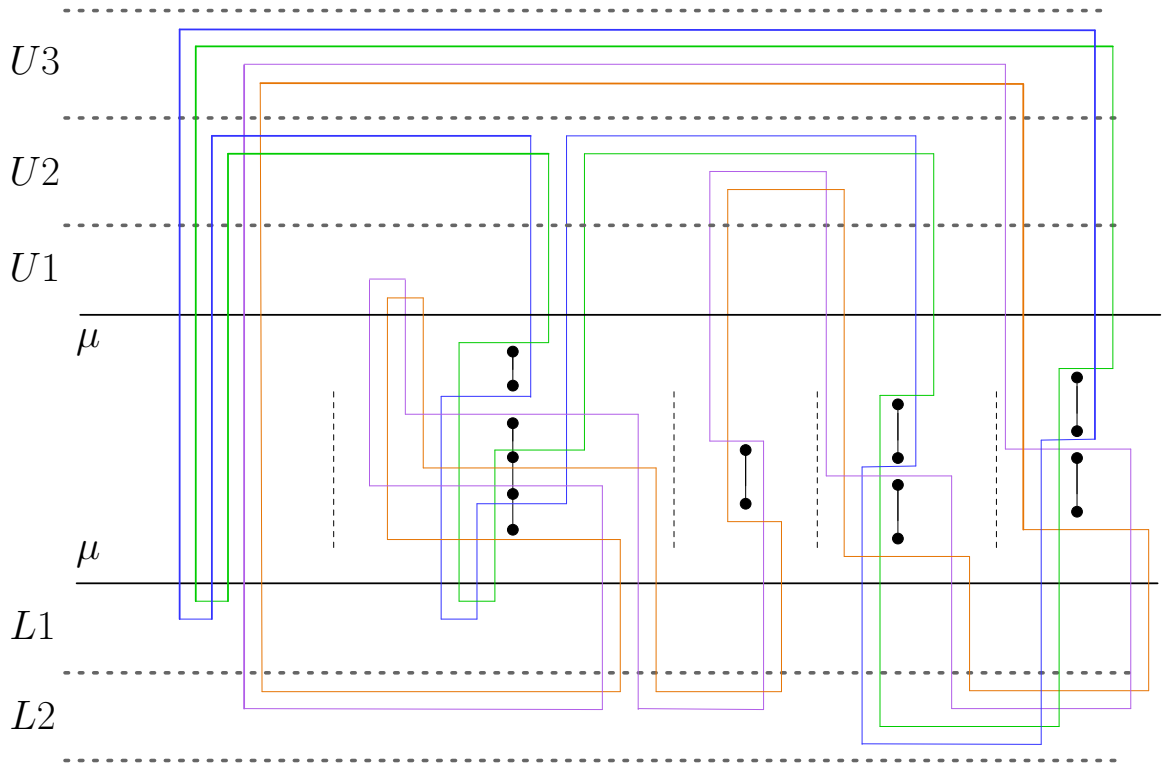}}
    \caption{Plumbing when $n_\alpha$ is odd and $n_\beta$ is even.}    \label{fig:alphaonly}
\end{figure}

\begin{figure}[H]
    \centering
    \scalebox{0.9}
    {\input{images/alpha_odd_beta_odd_plumb}}
    \caption{Plumbing when $n_\alpha$ is odd and $n_\beta$ is odd.}
    \label{fig:bothstrands}
\end{figure}

\begin{figure}[H]
    \centering
    \scalebox{0.8}{\input{images/alpha_even_beta_odd_plumb}}
    \caption{Plumbing when $n_\alpha$ is even and $n_\beta$ is odd.}
    \label{fig:betaonly}
\end{figure}

\begin{figure}[H]
    \centering
    \scalebox{0.75}{\input{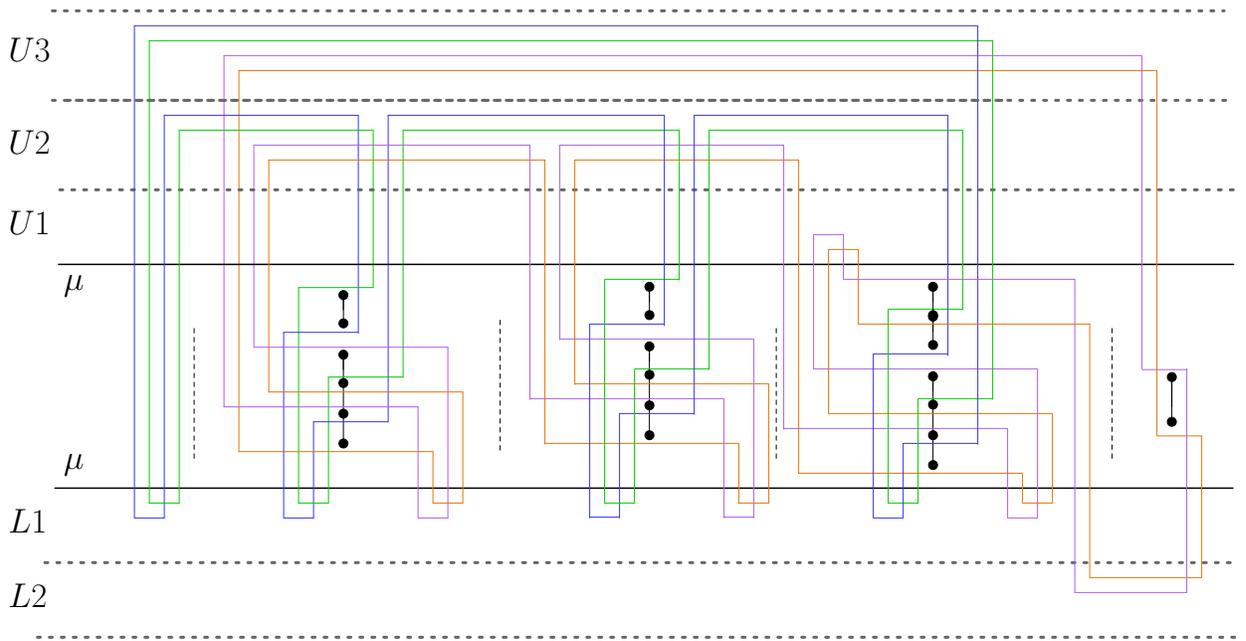}}
    \caption{Plumbing when $n_\alpha$ is even and $n_\beta$ is even.}
    \label{fig:nostrands}
\end{figure}

\subsubsection{Innermost bigon structure of bundle plumbings}

The following lemma will be used to show that in a slope $m$ plumbing sub-multiloop $\gamma_m$, there is no innermost embedded bigon of $\gamma_m$ both of whose marked points lie along a connecting arc of the plumbing.

\begin{lemma}[connecting arcs are staggered]
\label{lem:stagger}
    Let $m\in[0,\pi)$ and $\lambda_m$ and $\lambda'_m$ be connecting arcs in a slope $m$ plumbing sub-multiloop $\gamma_m$ that are both in the upper levels or both in the lower levels. Let $r$ and $r'$ be their associated radii, respectively. Let $a,b$ be the left and right $x$-coordinates of the connection points of $\lambda_m$ (relative to the vertical orientation of slope $m$), and let $a'$ and $b'$ be the left and right $x$-coordinates of the connection points of $\lambda'_m$. If $r>r'$, then $a<a'$ or $b>b'$.
\end{lemma}

\begin{proof}
 Suppose $r>r'$ and consider all possible combinations for the levels of $\lambda_m$ and $\lambda'_m$. We will consider the case where $\lambda_m$ and $\lambda'_m$ are in the upper levels $\{U1,U2,U3\}$. The following arguments are immediately adapted for the case where $\lambda_m$ and $\lambda'_m$ are in the lower levels $\{L1,L2\}$.

    \begin{enumerate}[leftmargin=0.5cm]
        \item \textbf{$U3$ vs $\{U1,U2,U3\}$.} \begin{itemize}[leftmargin=0cm]
            \item Suppose $\lambda_m\in U3$ and $\lambda'_m\in U1$ or $\lambda_m\in U3$ and $\lambda'_m\in U2$. Here $a<a'$  since $a$ is in the parity balancing block and $a'$ is not.

        \item Suppose $\lambda_m,\lambda'_m \in U3$. If both $\lambda_m$ and $\lambda'_m$ belong to $\alpha$ strands, then $\lambda_m$ belongs to a minus strand and $\lambda'_m$ belongs to a plus strand, so $a<a'$ since $a+\delta= a'$. The same argument applies if both $\lambda_m$ and $\lambda'_m$ are $\beta$  strands. On the other hand if $\lambda_m$ is a $\beta$ strand and $\lambda'_m$ is an $\alpha$ strand, then $a<a'$ since $a+\delta\leq a'$.

        \end{itemize}

        \item \textbf{$U2$ vs $\{U1,U2\}$.} 
        \begin{itemize}[leftmargin=0cm]
            \item Suppose $\lambda_m \in U2$ and $\lambda'_m \in U1$.
      In this case, $\lambda_m$ comes from an intra-bundle pairing and $\lambda'_m$ comes from an inter-bundle pairing. Thus $b'-a'=2\delta$ and $b-a>2\delta$, so $a<a'$ or $b>b'$.

        \item  Suppose $\lambda_m,\lambda'_m\in U2$. In this case, both $\lambda_m$ and $\lambda'_m$ come from an inter-bundle pairings. If both $\lambda_m$ and $\lambda'_m$ belong to $\alpha$ strands, then $\lambda_m$ belongs to a minus strand and $\lambda'_m$ belongs to a plus strand. Then either $\lambda_m$ and $\lambda'_m$ do not connect the same two blocks (in which case $[a,b]\cap[a',b']=\emptyset$), or they do connect the same two blocks, and then $a<a'$ since $a+\delta= a'$. The same argument applies if both $\lambda_m$ and $\lambda'_m$ belong to $\beta$ strands. On the other hand if $\lambda_m$ belongs to a $\beta$ strand and $\lambda'_m$ belongs to an $\alpha$ strand, we argue as follows. Since the edge gadget bundle associate to $b$ has at least one $\alpha$ edge, there is an upper terminal in that same bundle with an $\alpha$ label whose $x$-coordinate lies in $[a,b]$. This terminal cannot be skipped by $\lambda'_m$. Thus either $[a,b]\cap[a',b']=\emptyset$ or one of $a',b'$ lies in $[a,b]$.
 \end{itemize}
        \item \textbf{$U1$ vs $U1$.} 
        \begin{itemize}[leftmargin=0cm]
            \item Suppose $\lambda_m, \lambda'_m\in U1$. In this case, both $\lambda_m$ and $\lambda'_m$ come from intra-bundle pairings. If both $\lambda_m$ and $\lambda'_m$ belong to $\alpha$ strands, then $\lambda_m$ belongs to a minus strand and $\lambda'_m$ belongs to a plus strand, so either $[a,b]\cap[a',b']=\emptyset$ or  $a+\delta= a'$. The same argument applies if both $\lambda_m$ and $\lambda'_m$ belong to $\beta$ strands. On the other hand, if one of $\lambda_m$ and $\lambda'_m$ belongs to an $\alpha$ strand and the other belongs to a $\beta$ strand, then $[a,b]\cap[a',b']=\emptyset$. 
        \end{itemize}        
    \end{enumerate}
\noindent This completes the proof of the lemma. \end{proof}

\begin{lemma}[no accidental innermost bigons within a bundle plumbing]
\label{lem:bigons_within_plumbing_subloop}
Fix a slope $m\in[0,\pi)$ and consider a slope $m$ plumbing sub-multiloop $\gamma_m$ with $6$ strands as defined in Section \ref{sec:plumbing}. For $\epsilon>0$ sufficiently small (which we can compute in polynomial time), the only innermost embedded monorbigons of $\gamma_m$ are bigons $M_e$ around edges $e$ (as in Figure \ref{fig:edgegadget}) and bigons between $\mu$ and $\nu$ (as in Figure \ref{fig:edge_gadget_bundle}).
\end{lemma}

\begin{proof}
Let $B$ be an innermost bigon of $\gamma_m$ with marked points $x,y$ bounded by paths $\lambda,\lambda'$.

Suppose first that $\lambda$ or $\lambda'$ is a subarc of $\mu$ or $\nu$. In this case both $x$ and $y$ lie along $\mu$ or $\nu$, and $B$ contains a tooth bigon or subdivided tooth bigon. Thus we may assume that $\lambda$ and $\lambda'$ have labels among $\{\alpha_m^+,\alpha_m^-,\beta_m^+,\beta_m^-\}$.

Suppose now that $\lambda$ or $\lambda'$ intersects $\mu$ or $\nu$. Then it intersects $\mu=\partial\Ball(0,1)$ and there are teeth bigons to the left and right of this intersection points along $\mu$, one of which must be contained in $B$. Thus we may assume that $B$ lies entirely inside $\Ball(0,1)$, or outside it.

If $B$ lies entirely inside $\Ball(0,1)$, then a choice of $\epsilon$ sufficiently small ensures that all strands of $\{\alpha_m^+,\alpha_m^-,\beta_m^+,\beta_m^-\}$ are arbitrarily close to parallel to the slope $m$. This implies that each arc of $\alpha_m^\pm\cap\Ball(0,1)$ intersects each arc of $\beta_m^\pm\cap\Ball(0,1)$ at most once as in Figure \ref{fig:edge_gadget_bundle} at left. Thus $B$ must be an $M_e$ bigon.

Thus we may assume that $B$ lies entirely outside of $\Ball(0,1)$, and there are two cases according to whether $\lambda$ and $\lambda'$ have subarcs among the upper levels $\{U_1,U_2,U_3\}$ or lower levels $\{L_1,L_2\}$ (having subarcs among both is impossible since neither $\lambda$ nor $\lambda'$ intersect $\mu$).

Our construction implies that $\lambda$ and $\lambda'$ have at most two corners. Specifically, each has one of the following types:

\begin{itemize}[noitemsep]
    \item Type 0 (no corners): It is a subarc of a connecting arc,
    \item Type 1 (one corner): It consists of a one vertical arc and one subarc of a connecting arc (whose marked point on the vertical arc lies at a smaller radius than that on the connecting arc),
    \item Type 2 (two corners): It consists of two vertical arcs with an entire connecting arc in between them, with corners at the connection points.
\end{itemize}

Note first that Type 1 is in fact impossible, because if without loss of generality $\lambda$ is Type 1, there is no available type for $\lambda'$. It is also impossible for both $\lambda,\lambda'$ to be of Type 0, and for both of $\lambda,\lambda'$ to be of Type 2. In fact the only way for $\lambda$ and $\lambda'$ to form a bigon is for $\lambda$ to be of Type $2$ and $\lambda'$ to be of Type $0$. But then the connecting arc associated to $\lambda$ would be at  a larger radius than that of $\lambda'$, contradicting Lemma \ref{lem:stagger}. \end{proof}

\subsubsection{Proof of Theorem \ref{thm:20-hard}}

We are ready to prove the main theorem by reducing from \textsf{3C3PVC} to $\mathsf{SimplePin}(\Sigma,s)$ using this construction. We restate Theorem \ref{thm:20-hard}.

\begin{theorem}
For any fixed orientable surface $\Sigma$ and $s\geq 20$, $\mathsf{SimplePin}(\Sigma,s)$ is \textsf{NP}-complete.
\end{theorem}

\begin{proof}
    First note that that $\mathsf{SimplePin}(\Sigma,s)$ is in \textsf{NP}. Indeed, one may consult \cite{simon_stucky_2025pinningidealmultiloop} or appeal to Corollary \ref{cor:pinning-simple-easy}.

    To demonstrate that the problem is \textsf{NP}-hard, let $(G=(V,E),k)$ be an instance of \textsf{3C3PVC}, where $k\in \mathbb{Z}$. We will construct an instance $(\gamma,k')$ of $\mathsf{SimplePin}(\Sigma,s)$, where $k'=8\abs{E}+f+k$ for some $24\leq f\leq 72$ which depends on an embedding of $G$ in a way that will be described below. We will do this in polynomial time, and show that they are equivalent problems. An application of Theorem \ref{lem:pvc3c_hard} will finish the proof. 
    
    We may assume that $k>0$ and that $\Sigma=\R^2$ by Remark \ref{rmk:plane_enough}.
    
    \paragraph{Construct an embedding with six slopes.} From Lemma \ref{lem:embed_with_4_slopes}, we construct a straight-line embedding of $G$ in $\R^2$ in linear time in which the union of edge slopes is a $6$-element set $\{\frac{\pi}{4}, \frac{\pi}{2}, \frac{3\pi}{4},\phi_1,\phi_2,\phi_3\}$ for some $\phi_1,\phi_2,\phi_3\in [0,\pi)$, and each of $\phi_1$, $\phi_2$ and $\phi_3$ is realized exactly once by a single edge on the outer face of $G$.
    
    For each edge $e\in G$, let $L_e$ be the line containing $e$. Scale the embedding so that $G\subset \mathbb{B}(0,\frac{1}{2})$ and additionally that all intersection points between distinct lines $L_e$ and $L_{e'}$ are contained inside $\mathbb{B}(0,\frac{1}{2})$.
    
    \paragraph{Construct the multiloop via bundle plumbings.} Next, we apply the construction in Section \ref{sec:plumbing} once for each slope $m\in\{\frac{\pi}{4}, \frac{\pi}{2}, \frac{3\pi}{4},\phi_1,\phi_2,\phi_3\}$ to build a simple multiloop $\gamma$ with $4*6+2=26$ strands, choosing the same $\epsilon>0$ for all of them. We choose $\epsilon$ small enough so that all intersection points between distinct plumbing sub-multiloops $\gamma_m$ and $\gamma_{m'}$ lie inside $\Ball(0,\frac{1}{2})$. Additionally, we claim that we can compute $\epsilon$ in polynomial time small enough so that the following properties hold. Proofs will follow.
    
    \begin{enumerate}[leftmargin=0.5cm]
        \item         \label{cond:deformation_retraction} \textbf{Deformation retraction condition.} Let $M_e$ denote the bigon around edge $e$. For all $\{e_i\}_{i=1}^n\subset E$,
        \begin{equation*} 
            \bigcap_{i=1}^n M_{e_i} 
            \quad
            \text{retracts by deformation onto} 
            \quad
            \bigcap_{i=1}^n e_{i}.
        \end{equation*}

        This property ensures that the topology of gadget bigons mimics the topology of $G$.
        
        \item \label{cond:angle_condition} \textbf{Angle condition.} Let $\sigma_m,\sigma_{m'}\in\{\alpha_m^+,\alpha_m^-,\beta_m^+,\beta_m^-\}$ be any two strands associated to distinct slopes $m\neq m'$. Let $\lambda_m$ and $\lambda_{m'}$ be any two maximal paths in $\sigma_m\cap\mathbb{B}(0,\frac{1}{2})$ and $\sigma_{m'}\cap\mathbb{B}(0,\frac{1}{2})$, respectively. Then $\lambda_m\cap\lambda_{m'}$ is a single point.
        
        This property ensures that all intersections of $\gamma$ are transverse multiple points, and also helps control the mobidisc structure. 
    \end{enumerate}

    \emph{Proofs of claims.}
     \begin{enumerate}[leftmargin=0.5cm]
         \item \textbf{Proof of deformation retraction condition.} To achieve this, we must control both vertical and horizontal dimensions of gadget bigons.
        \begin{itemize}[leftmargin=0cm,noitemsep]       
        \item \textbf{Vertical separation.} For an edge $e$ of $G$, let $\ell(e)$ denote the length of the edge. We choose $\epsilon < \frac{1}{5}\min_{e\in E}\ell(e)$. 
        
        \item \textbf{Horizontal separation.} Consider the set of edges with slope $m$. These edges are partitioned into a set $\mathcal{C}_m$ of collinearity classes which define parallel lines. For distinct collinearity classes $C,C'\in\mathcal{C}_m$, let $\#C$ denote the number of edges in the collinearity class and let $d(C,C')$ denote the distance between these lines. Since the angles in Figure \ref{fig:edgegadget} tend to $0$ with $\epsilon$, we may choose $\epsilon$ small enough so that denoting $d_m= \min\{d(C,C') \colon  C,C'\in\mathcal{C}_m, C\neq C'\}$ and $c_m= 5\times\max\{\#C\colon C\in\mathcal{C}_m\}$ we have
        \begin{equation*}
            \epsilon\tan \epsilon < \min\left\{d_m/c_m \colon m\in\{\pi/4, \pi/2, 3\pi/4,\phi_1,\phi_2,\phi_3\}\right\}.
        \end{equation*}
        
        \end{itemize}
         \item \textbf{Proof of angle condition.}
         The paths $\lambda_m$ and $\lambda_{m'}$ intersect at least once since their endpoints are linked on $\mathbb{B}(0,\frac{1}{2})$. Since the angles in Figure \ref{fig:edgegadget} tend to $0$ with $\epsilon$, we can make all segments of $\lambda_m$ and $\lambda_{m'}$ arbitrarily close to slopes $m$ and $m'$, respectively, so that $\lambda_m$ and $\lambda_{m'}$ intersect transversely exactly once (for instance, by making $\max_{e\in E}\max_{i\in\{1,2,3,4\}}\theta_i<\frac{1}{3}\min\{\abs{m-m'}\colon m\neq m'\in\{\frac{\pi}{4}, \frac{\pi}{2}, \frac{3\pi}{4},\phi_1,\phi_2,\phi_3\}\}$). 
     \end{enumerate}
 
     \paragraph{Resolve multiple points.} Now we deal with isolated multiple points. We compute all intersection points of the multiloop in polynomial time, as well as the minimum distance $\epsilon'$ between them, and arbitrarily resolve multiple points locally (inside balls of radius $\epsilon'/2$, say) into sets of transverse double points. 
    
     \paragraph{Prune extra strands.} We now have a multiloop with exactly $26$ strands. To reach $s\geq 20$ strands, we proceed as follows. Since the three slopes $\phi_1$, $\phi_2$, and $\phi_3$ correspond to a unique edge on the outer face of $G$, we have $n_\alpha=1$ and $n_\beta=0$ for these slopes. In particular there are guaranteed to be no $\beta$ bigons for these slopes. We now remove the $\beta$ strands corresponding to these edges and the tooth bigons between $\mu$ and $\nu$ associated to them. This yields a multiloops with exactly $20$ strands.
     Now we add $s-20$ pairwise disjoint closed loops away from $\gamma$ and redefine $\gamma$ to be the resulting multiloop.     

     \paragraph{Compute the number of forced pins.} Recall that every edge of $G$ gives rise to $8$ bigons between $\mu$ and $\nu$, so that we will need to add $8\abs{E}$ to our potential pinning set of $\gamma$. 
     
     We now define the number $f$, which corresponds to the number of forced pins induced by the parity balancing procedure. Recall that for each slope $m\in\{\frac{\pi}{4}, \frac{\pi}{2}, \frac{3\pi}{4}\}$, there are $4$ ways that the parity balancing procedure may have been performed for the slope $m$ plumbing subloop. We count the number of extra bigons between $\mu$ and $\nu$ in each case, which is $4$ times the number of extra vertical arcs needed on the leftmost side of the plumbing construction.
     
     \begin{itemize}[noitemsep]
         \item If $n_\alpha$ is odd and $n_\beta$ is even, there are $6$ extra vertical arcs yielding $f_m=24$ extra bigons between $\mu$ and $\nu$ (see Figure \ref{fig:alpha_odd_beta_even}).
         \item If $n_\alpha$ is odd and $n_\beta$ is odd, there are $4$ extra vertical arcs yielding $f_m=16$ extra bigons between $\mu$ and $\nu$ (see Figure \ref{fig:alpha_odd_beta_odd}).
         \item If $n_\alpha$ is even and $n_\beta$ is odd, there are $2$ extra vertical arcs yielding $f_m=8$ extra bigons between $\mu$ and $\nu$ (see Figure \ref{fig:alpha_even_beta_odd}).
         \item If $n_\alpha$ is even and $n_\beta$ is even, there are $4$ extra vertical arcs yielding $f_m=16$ extra bigons between $\mu$ and $\nu$ (see Figure \ref{fig:alpha_even_beta_even}).
     \end{itemize}

    Now let $f=f_\frac{\pi}{4}+f_\frac{\pi}{2}+ f_\frac{3\pi}{4}$. Note that $f\equiv0\mod 8$ and $24\leq f\leq 72$.

    \paragraph{Define the problem instance and bound its size in memory.}
    Finally, let $k'=8\abs{E}+f+k$. This completes the polynomial-time construction of the $\mathsf{SimplePin}(\R^2,s)$ problem instance, $(\gamma,k')$.   
    
    Note that the size of $\gamma$ in memory is quadratic in the size of $G$. Indeed, $\gamma$ may be viewed as a union of  $2\abs{E}+\frac{f}{4}$  paths almost parallel to edges of $G$, plus $\abs{E}+\frac{f}{8}$ connecting arcs. They all intersect each other pairwise no more than $2$ times each, and each of the $2\abs{E}+\frac{f}{4}$ paths almost parallel to edges of $G$ gives rise to $4$ double points involving $\mu$ and $\nu$. Thus a crude bound is
    
    \[\#{\gamma}\leq 4(2\abs{E}+f/4)+2\binom{(2\abs{E}+f/4)+(\abs{E}+f/8)}{2}.\]

    \paragraph{Argue that the problems are equivalent.} It remains to argue that the graph $G$ has a vertex cover of size at most $k$ if and only if the simple multiloop $\gamma$ has a pinning set of size at most $k'$.    
    
    \paragraph{From pinning set to vertex cover.} Suppose $\gamma$ has a pinning set of size at most $k'$, and choose one such. Note that each of the $4\abs{E}+\frac{f}{2}$ regional teeth bigons must have a pin, and another $4\abs{E}+\frac{f}{2}$ pins are used for subdivided teeth. This accounts for at least $8\abs{E}+f$ pins. We construct a map from the remaining $\le k$ pins to vertices of $G$ as follows: If the pin is in no $M_e$, then choose any vertex; if the pin is only in one $M_e$, then choose either vertex associated to that bigon; if it is in more than one $M_e$, these intersect in a nonempty region containing a unique vertex of $G$ by the \textbf{deformation retraction condition} \eqref{cond:deformation_retraction}, which is the chosen vertex. Since every $M_e$ contains at least one pin by Corollary \ref{cor:pinning-simple-mobidiscs}, this yields a vertex cover of the graph using at most $k$ vertices.

     \paragraph{From vertex cover to pinning set.} Suppose the graph $G$ has a vertex cover of size at most $k$, and choose one such. We claim that every innermost bigon of $\gamma$ is either one of the bigons between $\mu$ and $\nu$, or one of the $M_e$ bigons. Indeed, Lemma \ref{lem:bigons_within_plumbing_subloop} implies that the only innermost bigons within each plumbing subloop are of this form, and since the slopes of strands of each plumbing subloop $\gamma_m$ are very close to $m$ within $\Ball(0,\frac{1}{2})$, the \textbf{angle condition} (\ref{cond:angle_condition}) and abundance of regional teeth bigons imply that there are no innermost bigons using two strands from two different plumbing subloops. Now pin $\gamma$ as follows. Each of the $8\abs{E}+f$ bigons between $\mu$ and $\nu$ receives a pin (choosing arbitrarily for subdivided teeth). For each vertex $v$ of the cover, we place a pin in the region $M_{e_1}\cap M_{e_2}\cap M_{e_3}$ where $e_1$, $e_2$ and $e_3$ are the $3$ edges connected to $v$. This region is nonempty and simply connected by the \textbf{deformation retraction condition} \eqref{cond:deformation_retraction}. Since we started with a vertex cover, all $M_e$ bigons are pinned with at most $k$ pins, so that $\gamma$ is pinned with at most $k'$ pins by Corollary \ref{cor:pinning-simple-mobidiscs}. \end{proof}   

     \subsubsection{Concluding remarks}
     \label{sec:conclusion}

     Recall the question raised by Theorems \ref{thm:3-easy} and \ref{thm:20-hard}: For a fixed surface $\Sigma$, what is the hardness threshold $s_{\textsf{mP}}=\max\{s\in\N\colon \mathsf{SimplePin}(\Sigma,s) \in \mathsf{P}\}$? Our theorems show (assuming \textsf{P}$\ne$\textsf{NP}) that $3\leq s_{\textsf{mP}}\leq 19$ and we conjecture  that $s_{\textsf{mP}}=3$ in Section \ref{sec:4-strands}.

     We conclude with a comment on the possibility of lowering the upper bound on $s_{\textsf{mP}}$ by lowering the value of $s$ appearing in the statement of Theorem \ref{thm:20-hard}.

     \begin{remark}[improving the upper bound on the number of strands] As mentioned in the introduction, the proof of Theorem \ref{thm:20-hard} uses $20$ strands for the following reason. For a $6$ slope drawing of a $3$-connected cubic planar graph $G$, it uses $4$ strands for each slope in $\{\frac{\pi}{4}, \frac{\pi}{2}, \frac{3\pi}{4}\}$, $2$ strands for each of the slopes among $\{\phi_1,\phi_2,\phi_e\}$ occurring exactly once on the outside face, and $2$ additional strands around the unit circle to anchor the construction. It is probably possible to use similar ideas as those we presented (at the cost of a more delicate construction or proof) to reduce the number of strands as low as $s\geq12$ (using only the $12$ strands corresponding to the slopes $\{\frac{\pi}{4}, \frac{\pi}{2}, \frac{3\pi}{4}\}$), thereby showing $s_{\textsf{mP}}\leq 11$. This is probably as far as one could hope to push the bound without significantly modifying our approach. Note that Uehara's construction in the proof of Lemma \ref{lem:pvc3c_hard} (originally given in \cite{uehara_cubic_NP_complete_96}) may be used to produce a $3$-connected cubic planar graph which also has girth at least $4$ (namely is \emph{triangle-free}); this additional restriction may be helpful for such endeavors.
\end{remark}

\section*{Acknowledgments}

We wish to thank David Eppstein for helpful comments about Lemma \ref{lem:embed_with_4_slopes}, and Ryuhei Uehara for permission to reproduce Figure \ref{fig:uehara_gadget} which first appeared in \cite{uehara_cubic_NP_complete_96}.

\newpage
\bibliographystyle{alpha}
\bibliography{paratex/bib}

\end{document}